\documentclass[11pt]{article}
\usepackage{a4}
\usepackage{amssymb}
\usepackage{amsmath}
\usepackage{amsthm}
\usepackage{enumerate}
\usepackage[pagebackref,colorlinks,citecolor=blue,linkcolor=blue]{hyperref}

\usepackage{geometry}
\numberwithin{equation}{section}
\theoremstyle{plain}
\newtheorem{theorem}[equation]{Theorem}
\newtheorem{corollary}[equation]{Corollary}
\newtheorem{lemma}[equation]{Lemma}
\newtheorem{proposition}[equation]{Proposition}

\theoremstyle{definition}
\newtheorem{definition}[equation]{Definition}

\newtheorem{example}[equation]{Example}
\newtheorem{remark}[equation]{Remark}

\numberwithin{equation}{section}

\newcommand{\R}{{\mathbb R}}
\newcommand{\N}{{\mathbb N}}

\newcommand{\Om}{\Omega}

\providecommand{\vint}[1]{\mathchoice
          {\mathop{\vrule width 5pt height 3 pt depth -2.5pt
                  \kern -9pt \kern 1pt\intop}\nolimits_{\kern -5pt{#1}}}
          {\mathop{\vrule width 5pt height 3 pt depth -2.6pt
                  \kern -6pt \intop}\nolimits_{\kern -3pt{#1}}}
          {\mathop{\vrule width 5pt height 3 pt depth -2.6pt
                  \kern -6pt \intop}\nolimits_{\kern -3pt{#1}}}
          {\mathop{\vrule width 5pt height 3 pt depth -2.6pt
                  \kern -6pt \intop}\nolimits_{\kern -3pt{#1}}}}

\newcommand{\eps}{\varepsilon}
\newcommand{\loc}{\mathrm{loc}}

\newcommand{\BV}{\mathrm{BV}}

\newcommand{\liploc}{\mathrm{Lip}_{\mathrm{loc}}}

\newcommand{\ch}{\text{\raise 1.3pt \hbox{$\chi$}\kern-0.2pt}}

\DeclareMathOperator{\Mod}{Mod}
\DeclareMathOperator{\capa}{Cap}
\DeclareMathOperator{\rcapa}{cap}

\DeclareMathOperator{\dist}{dist}

\DeclareMathOperator{\Lip}{Lip}

\DeclareMathOperator{\fint}{fine-int}

\begin{document}
\title{Capacities and $1$-strict subsets in metric spaces
\footnote{{\bf 2010 Mathematics Subject Classification}: 30L99, 31E05, 26B30.
\hfill \break {\it Keywords\,}: metric measure space,
capacity, strict subset, fine topology, function of bounded variation, pointwise approximation
}}
\author{Panu Lahti}
\maketitle

\begin{abstract}
In a complete metric space that is equipped with a doubling measure
and supports a Poincar\'e inequality, we study strict subsets, i.e.
sets whose variational capacity with respect to a larger reference set is finite.
Relying on the concept of fine topology,
we give a characterization of those strict subsets that are also
sets of finite perimeter, and then we apply this to the study
of condensers as well as BV capacities.
We also apply the theory to prove a pointwise approximation
result for functions of bounded variation.
\end{abstract}

\section{Introduction}

In potential theory, a set $A$ is said to be a $p$-\emph{strict} subset of a set
$D$ if the variational capacity $\rcapa_p(A,D)$ is finite, or equivalently
if there exists a Sobolev function $u$ with $u=1$ in $A$ and $u=0$
outside $D$.
In the case $1<p<\infty$, this concept has been considered in Euclidean spaces
in \cite{KiMa} and in the setting of more general metric measure spaces
in \cite{BBL-SS}.
The typical assumptions on a metric space, which we make also in the current paper, 
are that the space is complete, equipped with a doubling measure, and supports
a Poincar\'e inequality.

In the case $p=1$, $1$-strict subsets were studied, analogously to \cite{BBL-SS},
in \cite{L-NC}. However, these papers left largely open the question of how to
detect which sets are strict subsets. In the current paper we give a characterization
of those $1$-strict subsets that are also sets of finite perimeter, that is, their characteristic  functions are of bounded variation (BV).
The characterization involves the concepts of \emph{1-fine} interior and closure,
and the measure-theoretic interior $I_E$ of the set $E$; see Section \ref{sec:preliminaries}
for definitions.

\begin{theorem}\label{thm:strict subset theorem intro}
Let $D\subset X$
and let $E\subset X$ be a bounded set of finite perimeter with $I_E\subset D$.
Then $\rcapa_1(I_E,D)<\infty$ if and only if
\[
\capa_1(\overline{I_E}^1\setminus \fint D)=0.
\]
Moreover, then $\rcapa_1(I_E,D)\le C_a P(E,X)$ for a constant $C_a$ that depends only
on the doubling constant of the measure and the constants in the Poincar\'e
inequality.
\end{theorem}

In Example \ref{ex:non strict subset} we demonstrate that without the assumption
of finite perimeter, the theorem is not true.
After considering some preliminary results in Section \ref{sec:preliminary results},
we study $1$-strict subsets in Section
\ref{sec:strict subsets} and then we apply the theory to the study of \emph{condensers}
as well as $\BV$ versions of the variational capacity in
Section \ref{sec:applications}. These concepts have been studied previously
in e.g. \cite{HaSh}.
Perhaps the most important contribution of the current paper lies in our
careful analysis of the $1$-fine topology and the closely related notion
of \emph{quasiopen} sets. These have recently proved to be
very useful concepts (see especially \cite{L-Fedchar}) and we expect that a solid
understanding of their properties will contribute to future research as well.

As another application of our theory of $1$-strict subsets,
in Section \ref{sec:approximation}
we prove the following theorem on the approximation of BV functions
by means of Sobolev functions (often called Newton-Sobolev functions in the metric
space setting).

\begin{theorem}\label{thm:approximation from above intro}
Let $\Om\subset X$ be an open set with $\mu(\Om)<\infty$ and let $u\in\BV(\Om)$.
Then there exists a sequence $(w_i)\subset N^{1,1}(\Om)$
such that $w_i\to u$ in $L^1(\Om)$,
\[
\limsup_{i\to\infty}\int_{\Om}g_{w_i}\,d\mu\le \Vert Du\Vert(\Om)+C_a\Vert Du\Vert^j(\Om),
\]
where each $g_{w_i}$ is the minimal $1$-weak upper gradient of $w_i$ in $\Om$,
and $w_i(x)\ge u^{\vee}(x)$ and $w_i(x)\to u^{\vee}(x)$
for every $x\in \Om$.
\end{theorem}

Here the constant $C_a$ is the same as in Theorem \ref{thm:strict subset theorem intro}.
In Example \ref{ex:constant necessary} we show that the term $C_a\Vert Du\Vert^j(\Om)$
involving the jump part of the variation measure of $u$ is necessary.
Very recently, essentially the same result was proved in Euclidean spaces in
\cite[Proposition 7.3]{CCM}, based on an earlier
result \cite[Theorem 3.3]{CDLP}.
In Euclidean spaces the term $C_a\Vert Du\Vert^j(\Om)$
is not needed, but for us the existence of this term makes it necessary to
use rather different techniques in the proof, as we will discuss in 
Remark \ref{rmk:approximation thm}.

\section{Notation and definitions}\label{sec:preliminaries}

In this section we introduce the notation, definitions,
and assumptions that are employed in the paper.

Throughout this paper, $(X,d,\mu)$ is a complete metric space that is equip\-ped
with a metric $d$ and a Borel regular outer measure $\mu$ satisfying
a doubling property, meaning that
there exists a constant $C_d\ge 1$ such that
\[
0<\mu(B(x,2r))\le C_d\mu(B(x,r))<\infty
\]
for every ball $B(x,r):=\{y\in X:\,d(y,x)<r\}$.
We assume that $X$ consists of at least $2$ points.
Given a ball $B=B(x,r)$ and $\beta>0$, we sometimes abbreviate $\beta B:=B(x,\beta r)$;
note that in a metric space, a ball (as a set) does not necessarily have a unique center point and radius, but these will be prescribed for all the balls
that we consider.
When we want to state that a constant $C$
depends on the parameters $a,b, \ldots$, we write $C=C(a,b,\ldots)$.
When a property holds outside a set of $\mu$-measure zero, we say that it holds
almost everywhere, abbreviated a.e.

All functions defined on $X$ or its subsets will take values in $[-\infty,\infty]$.
As a complete metric space equipped with a doubling measure, $X$ is proper,
that is, closed and bounded sets are compact.
Given a $\mu$-measurable set $A\subset X$, we define $L^1_{\loc}(A)$
to be the class
of functions $u$ on $A$
such that for every $x\in A$ there exists $r>0$ such that $u\in L^1(A\cap B(x,r))$.
Other local spaces of functions are defined analogously.
For an open set $\Omega\subset X$,
a function is in the class $L^1_{\loc}(\Omega)$ if and only if it is in $L^1(\Om')$ for
every open $\Omega'\Subset\Omega$.
Here $\Omega'\Subset\Omega$ means that $\overline{\Omega'}$ is a
compact subset of $\Omega$.

For any set $A\subset X$ and $0<R<\infty$, the restricted Hausdorff content
of codimension one is defined by
\[
\mathcal{H}_{R}(A):=\inf\left\{ \sum_{j=1}^{\infty}
\frac{\mu(B(x_{j},r_{j}))}{r_{j}}:\,A\subset\bigcup_{j=1}^{\infty}B(x_{j},r_{j}),\,r_{j}\le R\right\}.
\]
We also allow finite coverings by interpreting $\mu(B(x,0))/0=0$.
The codimension one Hausdorff measure of $A\subset X$ is then defined by
\[
\mathcal{H}(A):=\lim_{R\rightarrow 0}\mathcal{H}_{R}(A).
\]

By a curve we mean a rectifiable continuous mapping from a compact interval of the real line into $X$.
The length of a curve $\gamma$
is denoted by $\ell_{\gamma}$. We will assume every curve to be parametrized
by arc-length, which can always be done (see e.g. \cite[Theorem~3.2]{Hj}).
A nonnegative Borel function $g$ on $X$ is an upper gradient 
of a function $u$
on $X$ if for all nonconstant curves $\gamma$, we have
\begin{equation}\label{eq:definition of upper gradient}
|u(x)-u(y)|\le \int_{\gamma} g\,ds:=\int_0^{\ell_{\gamma}} g(\gamma(s))\,ds,
\end{equation}
where $x$ and $y$ are the end points of $\gamma$.
We interpret $|u(x)-u(y)|=\infty$ whenever  
at least one of $|u(x)|$, $|u(y)|$ is infinite.
We also express inequality \eqref{eq:definition of upper gradient} by saying that
the pair $(u,g)$ satisfies the upper gradient inequality on the curve $\gamma$.
Upper gradients were originally introduced in \cite{HK}.

The $1$-modulus of a family of curves $\Gamma$ is defined by
\[
\Mod_{1}(\Gamma):=\inf\int_{X}\rho\, d\mu
\]
where the infimum is taken over all nonnegative Borel functions $\rho$
such that $\int_{\gamma}\rho\,ds\ge 1$ for every curve $\gamma\in\Gamma$.
A property is said to hold for $1$-almost every curve
if it fails only for a curve family with zero $1$-modulus. 
If $g$ is a nonnegative $\mu$-measurable function on $X$
and (\ref{eq:definition of upper gradient}) holds for $1$-almost every curve,
we say that $g$ is a $1$-weak upper gradient of $u$. 
By only considering curves $\gamma$ in a set $A\subset X$,
we can talk about a function $g$ being a ($1$-weak) upper gradient of $u$ in $A$.\label{curve discussion}

Given a $\mu$-measurable set $H\subset X$, we let
\[
\Vert u\Vert_{N^{1,1}(H)}:=\Vert u\Vert_{L^1(H)}+\inf \Vert g\Vert_{L^1(H)},
\]
where the infimum is taken over all $1$-weak upper gradients $g$ of $u$ in $H$.
Then we define the Newton-Sobolev space
\[
N^{1,1}(H):=\{u:\|u\|_{N^{1,1}(H)}<\infty\},
\]
which was first introduced in \cite{S}.
When $H$ is an open subset of $\R^n$, then for any $u\in N^{1,1}(H)$ the quantity $\Vert u\Vert_{N^{1,1}(H)}$
agrees with the classical Sobolev norm, see e.g. \cite[Corollary A.4]{BB}.
For any $\mu$-measurable function $u$ on a $\mu$-measurable set $H$, we also let
\[\label{definition of Dirichlet spaces}
\Vert u\Vert_{D^1(H)}:=\inf \Vert g\Vert_{L^1(H)},
\]
where the infimum is taken over all $1$-weak upper gradients $g$ of $u$ in $H$,
and then we define the Dirichlet space
\[
D^{1}(H):=\{u:\|u\|_{D^{1}(H)}<\infty\}.
\]
We understand Newton-Sobolev and Dirichlet functions to be defined at every $x\in H$
(even though $\Vert \cdot\Vert_{N^{1,1}(H)}$ is then only a seminorm).
It is known that for any $u\in D_{\loc}^{1}(H)$ there exists a minimal $1$-weak
upper gradient of $u$ in $H$, always denoted by $g_{u}$, satisfying $g_{u}\le g$ 
a.e. in $H$ for any $1$-weak upper gradient $g\in L_{\loc}^{1}(H)$
of $u$ in $H$, see \cite[Theorem 2.25]{BB}.

For any $D, H\subset X$, with $H$ $\mu$-measurable, the space of Newton-Sobolev functions with zero boundary values is defined as
\begin{equation}\label{eq:definition of N011}
N_0^{1,1}(D,H):=\{u|_{D\cap H}:\,u\in N^{1,1}(H)\textrm{ and }u=0\textrm { in }H\setminus D\}.
\end{equation}
This space is a subspace of $N^{1,1}(D\cap H)$ when $D$ is $\mu$-measurable, and it can always
be understood to be a subspace of $N^{1,1}(H)$.
If $H=X$, we omit it from the notation.
Similarly, the space of Dirichlet functions with zero boundary values is defined as
\[
D_0^{1}(D,H):=\{u|_{D\cap H}:\,u\in D^{1}(H)\textrm{ and }u=0\textrm { in }H\setminus D\}.
\]

We will assume throughout the paper that $X$ supports a $(1,1)$-Poincar\'e inequality,
meaning that there exist constants $C_P>0$ and $\lambda \ge 1$ such that for every
ball $B(x,r)$, every $u\in L^1_{\loc}(X)$,
and every upper gradient $g$ of $u$,
we have
\begin{equation}\label{eq:poincare inequality}
\vint{B(x,r)}|u-u_{B(x,r)}|\, d\mu 
\le C_P r\vint{B(x,\lambda r)}g\,d\mu,
\end{equation}
where 
\[
u_{B(x,r)}:=\vint{B(x,r)}u\,d\mu :=\frac 1{\mu(B(x,r))}\int_{B(x,r)}u\,d\mu.
\]

The $1$-capacity of a set $A\subset X$ is defined by
\[
\capa_1(A):=\inf \Vert u\Vert_{N^{1,1}(X)},
\]
where the infimum is taken over all functions $u\in N^{1,1}(X)$ satisfying
$u\ge 1$ in $A$. We know that $\capa_1$ is an outer capacity, meaning that
\[
\capa_1(A) = \inf \{\capa_1(W) : W \supset A,\ W \ \textrm{is open}\}
\]
for any $A \subset X$, see e.g. \cite[Theorem 5.31]{BB}.

The variational $1$-capacity of a set $A\subset D$
with respect to a set $D\subset X$ is defined by
\[
\rcapa_1(A,D):=\inf \int_X g_u \,d\mu,
\]
where the infimum is taken over functions $u\in N_0^{1,1}(D)$ satisfying
$u\ge 1$ in $A$, and $g_u$ is the minimal $1$-weak upper gradient of $u$ (in $X$).
For basic properties satisfied by capacities, such as monotonicity and countable subadditivity, see e.g. \cite{BB}.

If a property holds outside a set
$A\subset X$ with $\capa_1(A)=0$, we say that it holds $1$-quasieverywhere, or $1$-q.e.
If $H\subset X$ is $\mu$-measurable, then 
\begin{equation}\label{eq:quasieverywhere equivalence class}
v=0\ \textrm{ 1-q.e. in }H\textrm{ implies }\ \Vert v\Vert_{N^{1,1}(H)}=0,
\end{equation}
see \cite[Proposition 1.61]{BB}.
In particular, in the definition \eqref{eq:definition of N011} of the class
$N_0^{1,1}(D,H)$, we can equivalently require $u=0$ $1$-q.e. in $H\setminus D$,
and in the definition of the variational $1$-capacity we can require
$u\ge 1$ $1$-q.e. in $A$.

By \cite[Theorem 4.3, Theorem 5.1]{HaKi} we know that for any $A\subset X$,
\begin{equation}\label{eq:null sets of Hausdorff measure and capacity}
\capa_{1}(A)=0\quad
\textrm{if and only if}\quad\mathcal H(A)=0.
\end{equation}
We will use this fact numerous times in the paper.

\begin{definition}
	We say that a set $U\subset X$ is $1$-quasiopen
	if for every $\eps>0$ there exists an
	open set $G\subset X$ such that $\capa_1(G)<\eps$ and $U\cup G$ is open.
	
	Given a set $H\subset X$, we say that a function $u$
is $1$-quasicontinuous on $H$ if
for every $\eps>0$ there exists an open set $G\subset X$ such that $\capa_1(G)<\eps$
and $u|_{H\setminus G}$ is finite and continuous.
\end{definition}

It is a well-known fact that Newton-Sobolev functions are quasicontinuous
on open sets,
see \cite[Theorem 1.1]{BBS} or \cite[Theorem 5.29]{BB}.
The following more general fact is a special case of \cite[Theorem 1.3]{BBM}.

\begin{theorem}\label{thm:quasicontinuity on quasiopen sets}
	Let $U\subset X$ be $1$-quasiopen and let $u\in N^{1,1}_{\loc}(U)$.
	Then $u$ is $1$-quasicontinuous on $U$.
\end{theorem}

Next we present the definition and basic properties of functions
of bounded variation on metric spaces, following \cite{M}. See also e.g. \cite{AFP, EvGa, Fed, Giu84, Zie89} for the classical 
theory in the Euclidean setting.
Given an open set $\Om\subset X$ and a function $u\in L^1_{\loc}(\Om)$,
we define the total variation of $u$ in $\Om$ by
\[
\|Du\|(\Om):=\inf\left\{\liminf_{i\to\infty}\int_\Om g_{u_i}\,d\mu:\, u_i\in N^{1,1}_{\loc}(\Om),\, u_i\to u\textrm{ in } L^1_{\loc}(\Om)\right\},
\]
where each $g_{u_i}$ is the minimal $1$-weak upper gradient of $u_i$
in $\Om$.
(In \cite{M}, local Lipschitz constants were used in place of upper gradients, but the theory
can be developed similarly with either definition.)
We say that a function $u\in L^1(\Om)$ is of bounded variation, 
and denote $u\in\BV(\Om)$, if $\|Du\|(\Om)<\infty$.
For an arbitrary set $A\subset X$, we define
\[
\|Du\|(A):=\inf\{\|Du\|(W):\, A\subset W,\,W\subset X
\text{ is open}\}.
\]
In general, we understand the expression $\Vert Du\Vert(A)<\infty$ to mean that
there exists some open set $W\supset A$ such that $u$ is defined in $W$ with $u\in L^1_{\loc}(W)$ and $\Vert Du\Vert(W)<\infty$.

If $u\in L^1_{\loc}(\Om)$ and $\Vert Du\Vert(\Omega)<\infty$,
then $\|Du\|(\cdot)$ is
a Radon measure on $\Omega$ by \cite[Theorem 3.4]{M}.
A $\mu$-measurable set $E\subset X$ is said to be of finite perimeter if $\|D\ch_E\|(X)<\infty$, where $\ch_E$ is the characteristic function of $E$.
The perimeter of $E$ in $\Omega$ is also denoted by
\[
P(E,\Omega):=\|D\ch_E\|(\Omega).
\]
The measure-theoretic interior of a set $E\subset X$ is defined by
\begin{equation}\label{eq:measure theoretic interior}
I_E:=
\left\{x\in X:\,\lim_{r\to 0}\frac{\mu(B(x,r)\setminus E)}{\mu(B(x,r))}=0\right\},
\end{equation}
and the measure-theoretic exterior by
\begin{equation}\label{eq:measure theoretic exterior}
O_E:=
\left\{x\in X:\,\lim_{r\to 0}\frac{\mu(B(x,r)\cap E)}{\mu(B(x,r))}=0\right\}.
\end{equation}
The measure-theoretic boundary $\partial^{*}E$ is defined as the set of points
$x\in X$
at which both $E$ and its complement have strictly positive upper density, i.e.
\[
\limsup_{r\to 0}\frac{\mu(B(x,r)\cap E)}{\mu(B(x,r))}>0\quad
\textrm{and}\quad\limsup_{r\to 0}\frac{\mu(B(x,r)\setminus E)}{\mu(B(x,r))}>0.
\]
Note that the space $X$ is always partitioned into the disjoint sets
$I_E$, $O_E$, and $\partial^*E$.
For an open set $\Omega\subset X$ and a $\mu$-measurable set $E\subset X$ with $P(E,\Omega)<\infty$, we know that for any Borel set $A\subset\Omega$,
\begin{equation}\label{eq:def of theta}
P(E,A)=\int_{\partial^{*}E\cap A}\theta_E\,d\mathcal H,
\end{equation}
where
$\theta_E\colon \Om\to [\alpha,C_d]$ with $\alpha=\alpha(C_d,C_P,\lambda)>0$, see \cite[Theorem 5.3]{A1} 
and \cite[Theorem 4.6]{AMP}.
It follows that for any set $A\subset \Om$,
\begin{equation}\label{eq:consequence theta}
\alpha \mathcal H(\partial^*E\cap A)\le P(E,A)\le C_d \mathcal H(\partial^*E\cap A).
\end{equation}
The following coarea formula is given in \cite[Proposition 4.2]{M}:
if $\Omega\subset X$ is an open set and $u\in L^1_{\loc}(\Omega)$, then
\begin{equation}\label{eq:coarea}
\|Du\|(\Omega)=\int_{-\infty}^{\infty}P(\{u>t\},\Omega)\,dt,
\end{equation}
where we abbreviate $\{u>t\}:=\{x\in \Om:\,u(x)>t\}$.
If $\Vert Du\Vert(\Om)<\infty$, the above holds with $\Om$ replaced by
any Borel set $A\subset \Om$. By \cite[Proposition 3.8]{L-LSC}
this is true also for every $1$-quasiopen set $A\subset \Om$.

If $\Om\subset X$ is open and $u,v\in L^1_{\loc}(\Om)$, then
\begin{equation}\label{eq:variation of min and max}
\Vert D\min\{u,v\}\Vert(\Om)+\Vert D\max\{u,v\}\Vert(\Om)\le
\Vert Du\Vert(\Om)+\Vert Dv\Vert(\Om);
\end{equation}
for a proof see e.g. \cite[Proposition 4.7]{M}.

The $\BV$-capacity of a set $A\subset X$ is defined by
\[
\capa_{\BV}(A):=\inf \left(\Vert u\Vert_{L^1(X)}+\Vert Du\Vert(X)\right),
\]
where the infimum is taken over all $u\in\BV(X)$ such that $u\ge 1$ in a neighborhood of $A$.
As noted in \cite[Theorem 4.3]{HaKi}, for any $A\subset X$ we have
\begin{equation}\label{eq:Newtonian and BV capacities are comparable}
\capa_{\BV}(A)\le \capa_1(A).
\end{equation}

The lower and upper approximate limits of a function $u$ on an open set
$\Om$
are defined respectively by
\begin{equation}\label{eq:lower approximate limit}
u^{\wedge}(x):
=\sup\left\{t\in\R:\,\lim_{r\to 0}\frac{\mu(B(x,r)\cap\{u<t\})}{\mu(B(x,r))}=0\right\}
\end{equation}
and
\begin{equation}\label{eq:upper approximate limit}
u^{\vee}(x):
=\inf\left\{t\in\R:\,\lim_{r\to 0}\frac{\mu(B(x,r)\cap\{u>t\})}{\mu(B(x,r))}=0\right\}
\end{equation}
for $x\in \Om$.
The jump set of $u$ is then defined by
\[
S_u:=\{u^{\wedge}<u^{\vee}\}.
\]
Since we understand $u^{\wedge}$ and $u^{\vee}$ to be defined only on $\Om$,
also $S_u$ is understood to be a subset of $\Om$.
It is straightforward to check that $u^{\wedge}$ and $u^{\vee}$
are always Borel functions.

Unlike Newton-Sobolev functions, we understand $\BV$ functions to be
$\mu$-equivalence classes. To consider fine properties, we need to
consider the pointwise representatives $u^{\wedge}$ and $u^{\vee}$.

Recall that Newton-Sobolev functions are quasicontinuous;
$\BV$ functions have the following
quasi-semicontinuity property, which follows from \cite[Corollary 4.2]{L-SA},
which in turn is based on \cite[Theorem 1.1]{LaSh}.
The result was first proved in the Euclidean setting in
\cite[Theorem 2.5]{CDLP}.

\begin{proposition}\label{prop:quasisemicontinuity}
Let $u\in\BV(\Om)$ and let $\eps>0$. Then there exists an open set $G\subset \Om$
such that $\capa_1(G)<\eps$
and $u^{\wedge}|_{\Om\setminus G}$ is finite and lower semicontinuous and $u^{\vee}|_{\Om\setminus G}$ is finite and upper semicontinuous.
\end{proposition}

By \cite[Theorem 5.3]{AMP}, the variation measure of a $\BV$ function
can be decomposed into the absolutely continuous and singular part, and the latter
into the Cantor and jump part, as follows. Given an open set 
$\Omega\subset X$ and $u\in\BV(\Omega)$, we have for any Borel set $A\subset \Om$
\begin{equation}\label{eq:variation measure decomposition}
\begin{split}
\Vert Du\Vert(A) &=\Vert Du\Vert^a(A)+\Vert Du\Vert^s(A)\\
&=\Vert Du\Vert^a(A)+\Vert Du\Vert^c(A)+\Vert Du\Vert^j(A)\\
&=\int_{A}a\,d\mu+\Vert Du\Vert^c(A)+\int_{A\cap S_u}\int_{u^{\wedge}(x)}^{u^{\vee}(x)}\theta_{\{u>t\}}(x)\,dt\,d\mathcal H(x),
\end{split}
\end{equation}
where $a\in L^1(\Omega)$ is the density of the absolutely continuous part
and the functions $\theta_{\{u>t\}}\in [\alpha,C_d]$ 
are as in \eqref{eq:def of theta}.
Moreover, $\Vert Du\Vert^c(A)=0$ for any set $A$ of finite $\mathcal H$-measure.

Next we define the fine topology in the case $p=1$.
For the analogous definition and theory in the case $1<p<\infty$, see
e.g. the monographs \cite{AH,HKM,MZ} for the Euclidean case, as well as
\cite{BB,BBL-SS,BBL-CCK,BBL-WC} for the metric space setting.

\begin{definition}\label{def:1 fine topology}
We say that $A\subset X$ is $1$-thin at the point $x\in X$ if
\[
\lim_{r\to 0}r\frac{\rcapa_1(A\cap B(x,r),B(x,2r))}{\mu(B(x,r))}=0.
\]
We also say that a set $U\subset X$ is $1$-finely open if $X\setminus U$ is $1$-thin at every $x\in U$. Then we define the $1$-fine topology as the collection of $1$-finely open sets on $X$.

We denote the $1$-fine interior of a set $H\subset X$, i.e. the largest $1$-finely open set contained in $H$, by $\fint H$. We denote the $1$-fine closure of a set $H\subset X$, i.e. the smallest $1$-finely closed set containing $H$, by $\overline{H}^1$. The $1$-fine boundary of $H$
is $\partial^1 H:=\overline{H}^1\setminus \fint H$.
The $1$-base $b_1 H$ is defined as the set of points
where $H$ is \emph{not} $1$-thin.

We say that a function $u$ defined on a set $U\subset X$ is $1$-finely continuous at $x\in U$ if it is continuous at $x$ when $U$ is equipped with the induced $1$-fine topology on $U$ and $[-\infty,\infty]$ is equipped
with the usual topology.
\end{definition}

See \cite[Section 4]{L-FC} for discussion on this definition, and for a proof of the fact that the
$1$-fine topology is indeed a topology.
Using \cite[Proposition 6.16]{BB}, we see that a set $A\subset X$ is $1$-thin at $x\in X$
if and only if
\[
\lim_{r\to 0}\frac{\rcapa_1(A\cap B(x,r),B(x,2r))}
{\rcapa_1(B(x,r),B(x,2r))}=0.
\]

Now we list some known facts concerning the $1$-fine topology.
It is stated in \cite[Corollary 3.5]{L-Fed} that for any $A\subset X$,
\begin{equation}\label{eq:characterization of fine closure}
\overline{A}^1=A\cup b_1 A.
\end{equation}
By \cite[Lemma 3.1]{L-Fed} we have for any set $A\subset X$
\begin{equation}\label{eq:cap thickness points contain measure thickness points}
I_A\cup \partial^*A\subset b_1 A\subset \overline{A}^1.
\end{equation}
Note that by Lebesgue's differentiation theorem
(see e.g. \cite[Chapter 1]{Hei}),
for $\mu$-measurable $E\subset X$ we have $\mu(I_E\Delta E)=0$, where $\Delta$
denotes the symmetric difference. Thus the above implies
\begin{equation}\label{eq:meas th closure belongs to base}
I_E\cup\partial^*E\subset b_1 I_E\subset \overline{I_E}^1.
\end{equation}
By \cite[Proposition 3.3]{L-Fed},
\begin{equation}\label{eq:capacity of fine closure}
\capa_1(\overline{A}^1)=\capa_1(A)\quad\textrm{for any }A\subset X.
\end{equation}
\begin{theorem}[{\cite[Corollary 6.12]{L-CK}}]\label{thm:finely open is quasiopen and vice versa}
A set $U\subset X$ is $1$-quasiopen if and only if it is the union of a $1$-finely
open set and a $\mathcal H$-negligible set.
\end{theorem}

\begin{theorem}[{\cite[Theorem 5.1]{L-NC}}]\label{thm:fine continuity and quasicontinuity equivalence}
	A function $u$ on a $1$-quasiopen set $U$ is
	$1$-quasicontinuous on $U$ if and only if it is finite $1$-q.e. and $1$-finely
	continuous $1$-q.e. in $U$.
\end{theorem}

\emph{Throughout this paper we assume that $(X,d,\mu)$ is a complete metric space
	that is equipped with the doubling measure $\mu$ and supports a
	$(1,1)$-Poincar\'e inequality.}

\section{Preliminary results}\label{sec:preliminary results}

In this section we prove some preliminary results.

By \cite[Corollary 2.21]{BB} we know that if $H\subset X$ is a
$\mu$-measurable set and $v,w\in N^{1,1}(H)$, then
\begin{equation}\label{eq:upper gradient in coincidence set}
g_v=g_w\ \ \textrm{a.e. in}\ \ \{x\in H:\,v(x)=w(x)\},
\end{equation}
where $g_v$ and $g_w$ are the minimal $1$-weak upper gradients of $v$ and $w$ in $H$.

The following lemma is a special case of \cite[Lemma 1.52]{BB}.

\begin{lemma}\label{lem:sup is upper gradient}
Let $u_i$, $i\in\N$, be functions on a $\mu$-measurable set $H\subset X$
with $1$-weak upper gradients $g_i$.
Let $u:=\sup_{i\in\N} u_i$ and $g:=\sup_{i\in\N} g_i$,
and suppose that $\mu(\{u=\infty\})=0$.
Then $g$ is a $1$-weak upper gradient of $u$ in $H$.
\end{lemma}

\begin{lemma}\label{lem:capacity and perimeter}
Let $G\subset X$ and let $\eps>0$.
Then there exists an open set $W\supset G$ such that
$\capa_1(W)< C\capa_1(G)+\eps$ and $P(W,X)< C\capa_1(G)+\eps$,
for a constant $C=C(C_d,C_P,\lambda)$.
\end{lemma}

Recall that $C_d$, $C_P$, and $\lambda$ are the doubling constant of $\mu$
and the constants in the Poincar\'e inequality \eqref{eq:poincare inequality}.

\begin{proof}
By \eqref{eq:Newtonian and BV capacities are comparable} we have
$\capa_{\BV}(G)\le \capa_1(G)$.
By \cite[Lemma 3.2]{HaKi} (which is simply an application of Cavalieri's principle
and the coarea formula \eqref{eq:coarea})
we find a set $E\subset X$ containing a neighborhood of $G$ such that
\begin{equation}\label{eq:estimate of BV norm}
\mu(E)+P(E,X)<\capa_{\BV}(G)+\eps\le \capa_1(G)+\eps.
\end{equation}
By a suitable \emph{boxing inequality}, see \cite[Lemma 4.2]{HaKi}, we find balls $\{B(x_j,r_j)\}_{j=1}^{\infty}$ with $r_j\le 1$ covering
the measure-theoretic interior $I_E$, and thus also the set $G$, such that
\[
\sum_{j=1}^{\infty}\frac{\mu(B(x_j,r_j))}{r_j}\le C_B (\mu(E)+P(E,X))
\]
for some constant $C_B=C_B(C_d,C_P,\lambda)$.
For each $j\in\N$, by applying the coarea formula \eqref{eq:coarea} to the function
$u(y)= d(x_j,y)$,
we find a number $s_j\in [r_j,2r_j]$ such that
\begin{equation}\label{eq:perimeter estimate xj sj rj}
P(B(x_j,s_j),X) \le C_d\frac{\mu(B(x_j,r_j))}{r_j}.
\end{equation}
Define $1/s_j$-Lipschitz functions
\[
\eta_j(\cdot):=\max\left\{0,1-\frac{\dist(\cdot,B(x_j,s_j))}{s_j}\right\},\quad j\in\N,
\]
so that $0\le \eta_j\le 1$ on $X$, $\eta_j=1$
in $B(x_j,s_j)$ and
$\eta_j=0$ in $X\setminus B(x_j,2s_j)$.
Let $\eta:=\sup_{j\in\N} \eta_j$.
By \eqref{eq:upper gradient in coincidence set},
$\ch_{B(x_j,2s_j)}/s_j$ is a $1$-weak upper
gradient of $\eta_j$. Hence by Lemma \ref{lem:sup is upper gradient}
the minimal $1$-weak upper gradient of $\eta$ satisfies
$g_\eta\le\sum_{i=1}^{\infty}\ch_{B(x_j,2s_j)}/s_j$. Then
\begin{align*}
\int_X g_{\eta}\,d\mu
\le\sum_{j=1}^{\infty}\frac{\mu(B(x_j,2s_j))}{s_j}
&\le C_d^2\sum_{j=1}^{\infty}\frac{\mu(B(x_j,r_j))}{r_j}\\
&\le C_d^2 C_B (\mu(E)+P(E,X))\\
&< C_d^2 C_B (\capa_1(G)+\eps)\quad\textrm{by }\eqref{eq:estimate of BV norm}.
\end{align*}
Similarly we show that $\Vert \eta\Vert_{L^1(X)}\le C_d^2 C_B (\capa_1(G)+\eps)$.
Let $W:=\bigcup_{j=1}^{\infty}B(x_j,s_j)$.
Since $\eta=1$ in $W$, we get the estimate
\[
\capa_1(W)\le \Vert \eta\Vert_{N^{1,1}(X)}\le 2C_d^2 C_B (\capa_1(G)+\eps).
\]
Using the lower semicontinuity of perimeter with respect to $L^1$-convergence, as well as
\eqref{eq:variation of min and max}, we get
\begin{align*}
P(W,X)\le \sum_{j=1}^{\infty}P(B(x_j,s_j),X)
&\le C_d\sum_{j=1}^{\infty}\frac{\mu(B(x,r_j))}{r_j}\quad\textrm{by }\eqref{eq:perimeter estimate xj sj rj}\\
&\le C_d C_B(\mu(E)+P(E,X))\\
&< C_d C_B(\capa_1(G)+\eps).
\end{align*}
\end{proof}

Next we note that \emph{Federer's characterization} of sets of finite
perimeter holds also in metric spaces.

\begin{theorem}[{\cite[Theorem 1.1]{L-Fedchar}}]\label{thm:characterization}
	Let $\Omega\subset X$ be open, let $E\subset X$ be $\mu$-measurable, and suppose that
	$\mathcal H(\partial^*E \cap \Omega)<\infty$. Then $P(E,\Omega)<\infty$.
\end{theorem}

The converse holds by \eqref{eq:consequence theta}.

Recall the definitions of the measure-theoretic interior and exterior from
\eqref{eq:measure theoretic interior} and \eqref{eq:measure theoretic exterior}.

\begin{proposition}[{\cite[Proposition 4.2]{L-Fed}}]\label{prop:set of finite perimeter is quasiopen}
Let $\Omega\subset X$ be open and let $E\subset X$ be $\mu$-measurable with
$P(E,\Omega)<\infty$. Then $I_E\cap\Omega$ and $O_E\cap\Omega$ are $1$-quasiopen sets.
\end{proposition}

Now we generalize this proposition to quasiopen domains.

\begin{proposition}\label{prop:quasiopen sets in quasiopen sets}
Let $U\subset X$ be $1$-quasiopen and let
$E\subset X$ be $\mu$-measurable
with $\mathcal H(\partial^*E\cap U)<\infty$.
Then $I_E\cap U$ and $O_E\cap U$ are $1$-quasiopen sets.
\end{proposition}

\begin{proof}
We find a sequence of open sets $G_j\subset X$ such that
$U\cup G_j$ is open for each $j\in\N$ and $\capa_1(G_j)\to 0$ as $j\to\infty$.
By Lemma \ref{lem:capacity and perimeter}
we can assume that also $P(G_j,X)\to 0$, and so
$\mathcal H(\partial^*G_j)\to 0$ by \eqref{eq:consequence theta}.
It is straightforward to check that for each $j\in\N$
\[
\partial^*(E\cup G_j)\cap (U\cup G_j)
\subset (\partial^*E\cap U)\cup \partial^*G_j.
\]
Then
\[
\mathcal H(\partial^*(E\cup G_j)\cap (U\cup G_j))
\le \mathcal H(\partial^*E\cap U)+\mathcal H(\partial^*G_j)<\infty
\]
for each $j\in\N$.
By Theorem \ref{thm:characterization} we conclude that $P(E\cup G_j,U\cup G_j)<\infty$.
Thus each $I_{E\cup G_j}\cap (U\cup G_j)$ is $1$-quasiopen by Proposition
\ref{prop:set of finite perimeter is quasiopen}.
By \eqref{eq:capacity of fine closure} and the fact that $\capa_1$ is an
outer capacity, we can
take open sets $G_j'\supset \overline{G_j}^1$ such that still
$\capa_1(G_j')\to 0$.
By \eqref{eq:cap thickness points contain measure thickness points}
we have $I_{G_j}\cup \partial^* G_j\subset  \overline{G_j}^1\subset G_j'$,
and so
\[
I_{E}\setminus G_j'=I_{E\cup G_j}\setminus G_j'.
\]
Using this, we get
\[
(I_{E}\cap U)\cup G_j'=(I_{E}\cap (U\cup G_j))\cup G_j'
=(I_{E\cup G_j}\cap (U\cup G_j))\cup G_j',
\]
which is a union of a $1$-quasiopen and an open set for each $j\in\N$, and thus $1$-quasiopen.
It follows that $I_E\cap U$ is also $1$-quasiopen.
Similarly we show that $O_E\cap U$ is $1$-quasiopen.
\end{proof}

Recall the definitions concerning curves and $1$-modulus from page \pageref{curve discussion}.

\begin{proposition}\label{prop:ae curve goes through boundary quasiopen case}
Let $U\subset X$ be $1$-quasiopen and let
$E\subset X$ be $\mu$-measurable with $\mathcal H(\partial^*E\cap U)<\infty$.
Then for $1$-a.e. curve $\gamma$ in $U$
with $\gamma(0)\in I_E$ and $\gamma(\ell_{\gamma})\in O_E$, there exists
$t\in (0,\ell_{\gamma})$ such that $\gamma(t)\in\partial^*E$.
\end{proposition}

\begin{proof}
By Proposition \ref{prop:quasiopen sets in quasiopen sets} we know that
$I_E\cap U$ and $O_E\cap U$ are $1$-quasiopen sets.
By \cite[Remark 3.5]{S2} they are also $1$-path open, meaning that for $1$-a.e. curve $\gamma$,
$\gamma^{-1}(I_E\cap U)$ and $\gamma^{-1}(O_E\cap U)$ are relatively
open subsets of $[0,\ell_{\gamma}]$.
Let $\gamma$ be such a curve in $U$, with
$\gamma(0)\in I_E$ and $\gamma(\ell_{\gamma})\in O_E$.
Since $\gamma^{-1}(I_E\cap U)$ and $\gamma^{-1}(O_E\cap U)$ are
nonempty disjoint relatively open subsets of the connected set $[0,\ell_{\gamma}]$,
there necessarily exists a point $t\in (0,\ell_{\gamma})$ with
\[
t\notin \gamma^{-1}(I_E\cap U)\cup\gamma^{-1}(O_E\cap U),
\]
and so $t\in \gamma^{-1}(\partial^*E)$.
\end{proof}

In \cite[Example 5.4]{L-Fedchar} it is shown that the assumption
$\mathcal H(\partial^*E\cap U)<\infty$ cannot be removed.

\begin{lemma}\label{lem:fint of IE c}
Let $E\subset X$ be $\mu$-measurable.
Then $\fint I_E^c=\fint O_E$.
\end{lemma}
\begin{proof}
Note that $O_E\subset I_E^c$ and so $\fint O_E\subset \fint I_E^c$.
Conversely, we have
$\fint I_E^c= X\setminus \overline{I_E}^1$, and
by \eqref{eq:meas th closure belongs to base} we have
$X\setminus \overline{I_E}^1\subset X\setminus (I_E\cup\partial^*E)=O_E$.
Thus $\fint I_E^c\subset O_E$, and so $\fint I_E^c\subset \fint O_E$.
\end{proof}

By using a Lipschitz cutoff function like in the proof
of Lemma \ref{lem:capacity and perimeter}, it is easy to show that
for any ball $B(x,r)$ with $r\le 1$,
\begin{equation}\label{eq:capacity of ball}
\capa_1(B(x,r))\le 2\frac{\mu(B(x,2r))}{r}.
\end{equation}
It follows that for any $A\subset X$,
\begin{equation}\label{eq:capacity and Hausdorff measure}
\capa_1(A)\le 2C_d \mathcal H(A).
\end{equation}

\begin{lemma}\label{lem:set of finite Hausdorff measure quasiclosed}
Let $H\subset X$ be a Borel set with $\mathcal H(H)<\infty$. Then
$X\setminus H$ is a $1$-quasiopen set.
\end{lemma}
\begin{proof}
Let $\eps>0$. We find a closed set $K\subset H$ such that
$\mathcal H(H\setminus K)<(2C_d)^{-1}\eps$ (see e.g. \cite[Proposition 3.3.37]{HKST15}).
By \eqref{eq:capacity and Hausdorff measure}, $\capa_1(H\setminus K)<\eps$,
and then since $\capa_1$ is an outer capacity, we find an open set $G\supset H\setminus K$
such that $\capa_1(G)<\eps$. Now $(X\setminus H)\cup G=(X\setminus K)\cup G$ is an open set. 
\end{proof}

Given a closed set $F\subset X$, one can of course always find open sets
$W_1\supset W_2\supset \ldots\supset F$ such that $\bigcap_{j=1}^{\infty}W_j=F$.
For $1$-quasiopen sets we have the following analog of this fact.

\begin{lemma}\label{lem:using quasi Lindelof}
Let $F\subset X$ such that $X\setminus F$ is $1$-quasiopen. Then there exist open sets
$W_1\supset W_2\supset \ldots\supset F$ such that
\[
\capa_1\left(\bigcap_{j=1}^{\infty}W_j\setminus F\right)=0.
\]
\end{lemma}

By Lemma \ref{lem:set of finite Hausdorff measure quasiclosed}, $F$
can in particular be any Borel set of finite $\mathcal H$-measure.

\begin{proof}
	For each $j\in\N$ we find an open set $G_j\subset X$ such that
	$F\setminus G_j$ is a closed set, and $\capa_1(G_j)\to 0$.
	Then for each $j\in\N$ we find open sets
	\[
	V_{j1}\supset V_{j2}\supset \ldots\supset F\setminus G_j
	\]
	such that $F\setminus G_j=\bigcap_{i=1}^{\infty}V_{ji}$.
	Define $W_j:=\bigcap_{k=1}^j (V_{kj}\cup G_k)$ for each $j\in\N$.
	These form a decreasing sequence of open sets containing $F$, and for each $N\in\N$,
	\begin{align*}
	\capa_1\left(\bigcap_{j=1}^{\infty}W_j\setminus F\right)
	&\le \capa_1\left(\bigcap_{j=N}^{\infty}W_j\setminus F\right)\\
	&\le \capa_1\left(\bigcap_{j=N}^{\infty}(V_{Nj}\cup G_N)\setminus F\right)\le\capa_1(G_N),
	\end{align*}
	since $\bigcap_{j=N}^{\infty}V_{Nj}=F\setminus G_N$.
Letting $N\to\infty$, we get the result.
\end{proof}

Finally we prove the following absolute continuity.

\begin{lemma}\label{lem:absolute continuity of H wrt capa}
Let $H\subset X$ with $\mathcal H(H)<\infty$. Then for every
$\eps>0$ there exists $\delta>0$ such that if $A\subset X$ with $\capa_1(A)<\delta$,
then $\mathcal H(H\cap A)<\eps$.
\end{lemma}
\begin{proof}
Suppose by contradiction that there exists $\eps>0$ and a sequence
of sets $A_j\subset X$, $j\in\N$, such that $\capa_1(A_j)< 2^{-j}$
but $\mathcal H(H\cap A_j)\ge \eps$.
Since $\capa_1$ is an outer capacity, we can assume that the sets $A_j$ are open.
Then defining
\[
A:=\bigcap_{k=1}^{\infty}\bigcup_{j\ge k}A_j,
\]
we have $\capa_1(A)=0$. However, since the sets $\bigcup_{j\ge k}A_j$ constitute
a decreasing sequence of Borel sets and since the restriction
$\mathcal H|_{H}(\cdot):=\mathcal H(H\cap \cdot)$ is a Borel outer
measure (see e.g. \cite[Lemma 3.3.13]{HKST15}), we have
\[
\mathcal H(H\cap A)=\lim_{k\to\infty}\mathcal H\Big(H\cap \bigcup_{j\ge k}A_j\Big)
\ge\eps,
\]
which is a contradiction
by \eqref{eq:null sets of Hausdorff measure and capacity}.
\end{proof}

\section{Strict subsets}\label{sec:strict subsets}

In this section we study $1$-strict subsets.

\begin{definition}
	A set $A\subset D$ is a \emph{1-strict subset} of $D$ if there is a function
	$\eta\in N_0^{1,1}(D)$ such that $\eta=1$ in $A$.
\end{definition}

By \cite[Proposition 3.3]{L-Fed} we know that if $A$ is a $1$-strict subset of $D$,
then $\capa_1(\overline{A}^1\setminus \fint D)=0$.
Now we show that
this holds also with the ambient space $X$ replaced by a more general quasiopen
set $U$.
Note that $1$-quasiopen sets are $\mu$-measurable
by \cite[Lemma 9.3]{BB-OD}.

\begin{proposition}\label{prop:one direction}
Let $A\subset D$ and let $U\subset X$ be a $1$-quasiopen set,
and suppose that
there exists $\rho\in N_0^{1,1}(D,U)$ with $\rho=1$ in $A\cap U$.
Then
\begin{equation}\label{eq:IE outside fint U}
\capa_1((\overline{A}^1\setminus \fint D)\cap U)=0.
\end{equation}
\end{proposition}

\begin{proof}
The function $\rho$ is $1$-quasicontinuous on $U$ by
Theorem \ref{thm:quasicontinuity on quasiopen sets}, and thus
$1$-finely continuous (with respect to the induced $1$-fine topology on $U$) at $1$-q.e. point in $U$
by Theorem \ref{thm:fine continuity and quasicontinuity equivalence}.
Now for $1$-q.e. $x\in \overline{A}^1\cap U$ we have either
$x\in A$ or $x\in b_1 A$ by \eqref{eq:characterization of fine closure},
and also $x\in \fint U$ by
Theorem \ref{thm:finely open is quasiopen and vice versa}.
Then either $x\in A$ or
$x\in b_1 (A\cap U)$. If $\rho$ is $1$-finely continuous
at $x$, it follows that
$\rho(x)=1$. In conclusion, $\rho=1$ $1$-q.e. in $\overline{A}^1\cap U$.

Analogously, from the fact that $\rho=0$ in $U\setminus D$
we get $\rho=0$ $1$-q.e. in $\overline{X\setminus D}^1\cap U=U\setminus \fint D$,
and then
\eqref{eq:IE outside fint U} follows.
\end{proof}

Now we note that the converse to Proposition \ref{prop:one direction}
is not true.
\begin{example}\label{ex:non strict subset}
Let $X=\R^2$ (unweighted, i.e. equipped with the usual $2$-dimensional Lebesgue measure).
We will choose a compact subset $K$ of a $1$-finely
open set $D$ such that $K$ is not a $1$-strict subset of $D$ (note that $K=\overline{K}^1$).
First denote the unit square by $Q:=[0,1]\times [0,1]$.
Define the following ``gratings'' that are compact subsets of $Q$:
\[
H_j:=\bigcup_{k=0}^{2^j} \{k 2^{-j}\}\times [0,1],\quad j\in\N.
\]
Given any set $A\subset \R^2$ and $a>0$, $b\in\R^2$, scaling and translation are given by
\[
aA+b:=\{ax+b:\,x\in A\}.
\]
Now consider the complement of the union of scaled and shifted ``gratings''
\[
D:=\R^2\setminus \bigcup_{j=1}^{\infty}(2^{-2j}H_{2j}+(2^{-j},0))
\]
All points in $D$ are interior points except the origin $0$.
We note that for every $r>0$ and every set
$2^{-2j}H_{2j}+(2^{-j},0)$ that intersects $B(0,r)$,
we have
\[
\rcapa_1(2^{-2j}H_{2j}+(2^{-j},0),B(0,2r))\le \rcapa_1(2^{-2j}Q,B(0,2r))= 2^{-4j}.
\]
It follows that for every $0<r<1/4$ ($\mathcal L^2$ denotes the $2$-dimensional Lebesgue measure,
and $\lfloor a \rfloor$ is the largest integer at most $a\in\R$)
\begin{align*}
&r\frac{\rcapa_1(B(0,r)\setminus D,B(0,2r))}{\mathcal L^2(B(0,r))}\\
&\qquad\qquad\le \frac{r}{\mathcal L^2(B(0,r))}\sum_{j=\lfloor-\log r/\log 2\rfloor-1}^{\infty}\rcapa_1(2^{-2j}H_{2j}+(2^{-j},0),B(0,2r))\\
&\qquad\qquad\le \frac{r}{\mathcal L^2(B(0,r))}\sum_{j=\lfloor-\log r/\log 2\rfloor-1}^{\infty}2^{-4j}\\
&\qquad\qquad\le \frac{2r}{\mathcal L^2(B(0,r))}2^{-4\lfloor-\log r/\log 2\rfloor+4}\\
&\qquad\qquad\le \frac{2^9 r}{\mathcal L^2(B(0,r))}r^4\to 0\quad\textrm{as }r\to 0.
\end{align*}
Thus the set $D$ is $1$-finely open.
Now define
\[
K_j:=\bigcup_{k=0}^{2^j-1} \{(k+1/2) 2^{-j}\}\times [0,1],\quad j\in\N,
\]
which are also compact subsets of the unit square, and
\[
K:=\bigcup_{j=1}^{\infty}(2^{-2j}K_{2j}+(2^{-j},0))\cup\{0\},
\]
which is a compact subset of $D$.
Let $u\in N_0^{1,1}(D)$ be a function with $u=1$ in $K$, and let $g$ be any
upper gradient of $u$. Now for every $j\in\N$,
\[
\Vert u\Vert_{N^{1,1}(2^{-2j}Q+(2^{-j},0))}\ge \int_{2^{-2j}Q+(2^{-j},0)}g\,d\mathcal L^2
\ge 2^{-2j}\cdot 2\cdot (2^{2j}-1)\ge 1.
\]
Thus $\Vert u\Vert_{N^{1,1}(\R^2)}=\infty$, and so $K$ is not a $1$-strict subset of $D$.
\end{example}

In \cite[Theorem 4.3]{L-NC} it is shown that when $A$
is a point in $\fint D$, then $A$ is a $1$-strict subset of $D$.
Now our goal will be to show that despite Example \ref{ex:non strict subset},
there are many other $1$-strict
subsets $A$ of $D$.
Our first result in this direction is the following.

\begin{lemma}\label{lem:Sobolev functions for finite Hausdorff measure 1}
Let $W\subset  X$ be an open set, let
$H\subset W$ with
$\mathcal H(H)<\infty$, and
let $\eps>0$.
Then there exists
$\eta\in N_0^{1,1}(W)$ with $0\le\eta\le 1$ on $X$, $\eta=1$ in
a neighborhood of $H$, and
\[
\int _X \eta\,d\mu< \eps\quad\textrm{and}\quad
\int_X g_{\eta}\,d\mu< C_d\mathcal H(H)+\eps.
\]
Moreover, $\capa_1(\{\eta>0\})<2C_d^2(\mathcal H(H)+\eps)$
and $\capa_1(\overline{\{\eta>0\}}^1\setminus W)=0$.
\end{lemma}

\begin{proof}
Let
\[
W_{\delta}:=\{x\in W:\, \dist(x,X\setminus W)>\delta\},\quad \delta>0.
\]
Let $V_1:= W_{2^{-1}}$ and for each $j=2,3,\ldots$,
let $V_j:= W_{2^{-j}}\setminus W_{2^{-j+1}}$. Then $W=\bigcup_{j=1}^{\infty}V_j$.
For each $j\in\N$, take a collection of balls
$\{B_{jk}=B(x_{jk},r_{jk})\}_{k=1}^{\infty}$ covering
$H\cap V_j$ with
$r_{jk}\le 2^{-j-2}\eps(C_d\mathcal H(H)+\eps+1)^{-1}$ and
\begin{equation}\label{eq:choice of balls Bjk}
\sum_{k=1}^{\infty}\frac{\mu(B_{jk})}{r_{jk}}<\mathcal H(H\cap V_j)
+\frac{2^{-j}\eps}{C_d}.
\end{equation}
We can assume that $B_{jk}\cap (H\cap V_j)\neq \emptyset$ for all $k\in\N$.
Define Lipschitz functions
\[
\eta_{jk}:=\max\left\{0,1-\frac{\dist(\cdot,B_{jk})}{r_{jk}}\right\},\quad j,k\in\N,
\]
so that $\eta_{jk}=1$ in $B_{jk}$ and $\eta_{jk}=0$ in $X\setminus 2 B_{jk}$.
By \eqref{eq:upper gradient in coincidence set}, each
$\eta_{jk}$ has a $1$-weak upper gradient $\ch_{2B_{jk}}/r_{jk}$.
Let $\eta:=\sup_{j,k\in\N}\eta_{jk}$.
Then $\eta=1$ in a neighborhood of $H$ and $\eta=0$ in $X\setminus W$.
By Lemma \ref{lem:sup is upper gradient},
\begin{equation}\label{eq:g eta estimate by H H}
\begin{split}
\int_X g_{\eta}\,d\mu
\le \sum_{j,k=1}^{\infty}\int_X g_{\eta_{jk}}\,d\mu
&\le \sum_{j,k=1}^{\infty}\frac{\mu(2 B_{jk})}{r_{jk}}\\
&< C_d\sum_{j=1}^{\infty}\left(\mathcal H(H\cap V_j)+\frac{2^{-j}\eps}{C_d}\right)\quad\textrm{by }\eqref{eq:choice of balls Bjk}\\
&=C_d\mathcal H(H)+\eps,
\end{split}
\end{equation}
as desired.
Similarly,
\[
\int_X \eta\,d\mu
\le \sum_{j,k=1}^{\infty}\mu(2 B_{jk})
\le \eps (C_d\mathcal H(H)+\eps+1)^{-1}\sum_{j,k=1}^{\infty}\frac{\mu(2 B_{jk})}{r_{jk}}< \eps.
\]
In conclusion, $\eta\in N^{1,1}(X)$ and then in fact $\eta\in N_0^{1,1}(W)$.

We can define a function $\rho$ analogously to $\eta$, but using
the collections of balls $\{2B_{jk}\}_{k=1}^{\infty}$
in place of $\{B_{jk}\}_{k=1}^{\infty}$.
We obtain $\rho=1$ in $\{\eta>0\}$ and then
\[
\capa_1\left(\{\eta>0\}\right)
\le \int_X \rho\,d\mu+\int_X g_{\rho}\,d\mu
\le 2C_d^2(\mathcal H(H)+\eps).
\]
Moreover, by the characterization of the fine closure
\eqref{eq:characterization of fine closure}, we get
$\overline{\{\eta>0\}}^1\setminus W
\subset\overline{\bigcup_{j=N}^{\infty}\bigcup_{k=1}^{\infty}2B_{jk}}^1$ 
for any $N\in\N$,
and so
\begin{align*}
\capa_1(\overline{\{\eta>0\}}^1\setminus W)
&\le \capa_1\left(\overline{\bigcup_{j=N}^{\infty}\bigcup_{k=1}^{\infty}2B_{jk}}\right)\\
&=\capa_1\left(\bigcup_{j=N}^{\infty}\bigcup_{k=1}^{\infty}2B_{jk}\right)
\quad\textrm{by }\eqref{eq:capacity of fine closure}\\
&\le 2\sum_{j=N}^{\infty}\sum_{k=1}^{\infty}\frac{\mu(4B_{jk})}{2r_{jk}}\quad\textrm{by }\eqref{eq:capacity of ball}\\
&\to 0
\end{align*}
as $N\to \infty$, since we had
$\sum_{j,k=1}^{\infty}\mu(2 B_{jk})/r_{jk}<\infty$
by \eqref{eq:g eta estimate by H H}.
\end{proof}

\begin{lemma}[{\cite[Lemma 3.3]{L-LSC}}]\label{lem:capacity and Newtonian function and bigger set}
	Let $G\subset X$ and $\eps>0$. Then there exists an open set 
	$G'\supset G$ with $\capa_1(G')< C_1(\capa_1(G)+\eps)$ and a
	function $\rho\in N^{1,1}_0(G')$ with $0\le\rho\le 1$ on $X$, $\rho=1$ in $G$, and $\Vert \rho\Vert_{N^{1,1}(X)}< C_1(\capa_1(G)+\eps)$,
	for some constant $C_1=C_1(C_d,C_P,\lambda)\ge 1$.
\end{lemma}

The following proposition says that a subset of finite Hausdorff measure
of a $1$-quasiopen set is always a $1$-strict subset.

\begin{proposition}\label{prop:Sobolev functions for finite Hausdorff measure 2}
Let $U\subset  X$ be $1$-quasiopen and let $F\subset U$ with
 $\mathcal H(F)<\infty$. Let $0<\eps<1$. Then there exists
$\eta\in N_0^{1,1}(U)$ with $0\le\eta\le 1$ on $X$, $\eta=1$ in a
$1$-quasiopen set containing $F$, and
\[
\int _X \eta\,d\mu< \eps\quad\textrm{and}\quad
\int_X g_{\eta}\,d\mu<C_d\mathcal H(F)+\eps.
\]
Moreover, $\eta=0$ in a $1$-quasiopen set containing $X\setminus U$.
\end{proposition}

\begin{proof}
For each $j\in\N$, by Lemma \ref{lem:absolute continuity of H wrt capa}
there exists $0<\delta_j<1$ such that if $A\subset X$ with $\capa_1(A)<\delta_j$,
then $\mathcal H(F\cap A)<2^{-j-3}\eps/C_d^2$.
For each $j\in\N$ we find an open set $G_j\subset X$ such that
$U\cup G_j$ is open and $\capa_1(G_j)< 2^{-j-1}\eps\delta_j/C_1$.
By Lemma \ref{lem:capacity and Newtonian function and bigger set}
we then find an open set $G_j'\supset G_j$ with $\capa_1(G_j')<2^{-j-1}\eps\delta_j$
and a function $\rho_j\in N_0^{1,1}(G_j')$ such that
$0\le \rho_j\le 1$ on $X$, $\rho_j=1$ in $G_j$,
and $\Vert \rho_j\Vert_{N^{1,1}(X)}<2^{-j-1} \eps\delta_j$.

By \eqref{eq:capacity of fine closure}, also
$\capa_1(\overline{G_j'}^1)<2^{-j-1}\eps\delta_j$ for each $j\in\N$, and so
$\mathcal H(F\cap \overline{G_j'}^1)<2^{-j-3}\eps/C_d^2$.
Let also $G_0':=X$.
For each $j\in\N$, apply
Lemma \ref{lem:Sobolev functions for finite Hausdorff measure 1} with
the choices $H=F\cap \overline{G_{j-1}'}^1$ and
$W=U\cup G_{j}$ to find
a function $\eta_j\in N_0^{1,1}(U\cup G_{j})$
such that $0\le \eta_j\le 1$ on $X$,
$\eta_j=1$ in an open set $W_j\supset F\cap \overline{G_{j-1}'}^1$, and
\[
\int_X \eta_j\,d\mu<2^{-j}\eps\quad\textrm{and}\quad\int_X g_{\eta_j}\,d\mu< 
C_d\mathcal H \left(F\cap \overline{G_{j-1}'}^1\right)+2^{-j-2}\eps.
\]
Moreover, Lemma \ref{lem:Sobolev functions for finite Hausdorff measure 1}
further gives for $j=2,3,\ldots$ (note that the $\eps$ in that lemma can be chosen as small
as needed)
\begin{equation}\label{eq:capacity of eta j set}
\capa_1(\{\eta_j>0\})<2C_d^2(\mathcal H(F\cap \overline{G_{j-1}'}^1)+2^{-j-2}\eps/C_d^2)
\le 2^{-j-1}\eps+2^{-j-1}\eps= 2^{-j}\eps.
\end{equation}
Now
\[
\int_X g_{\eta_1}\,d\mu< C_d\mathcal H(F)+2^{-3}\eps
\]
and for $j=2,3\ldots$,
\[
\int_X g_{\eta_j}\,d\mu< C_d \mathcal H\big(F\cap \overline{G_{j-1}'}^1\big)
+2^{-j-2}\eps<2^{-j-2}\eps+2^{-j-2}\eps=  2^{-j-1}\eps.
\]
Then let $\eta'_j:=\eta_j(1-\rho_j)$ for each $j\in\N$.
Now we have
\begin{equation}\label{eq:eta prime one}
\eta'_j=1 \textrm{ in }W_j\setminus G_j'\supset F\cap \overline{G_{j-1}'}^1\setminus G_j',
\end{equation}
$\eta_j'=0$ in $X\setminus U$, and
by the Leibniz rule \cite[Theorem 2.15]{BB},
\begin{equation}\label{eq:estimate for N11 norm of eta 1 prime}
\int_{X} g_{\eta_1'}\,d\mu
\le \int_{X} g_{\eta_1}\,d\mu +\int_{X} g_{\rho_1}\,d\mu
< C_d\mathcal H(F)+2^{-1}\eps
\end{equation}
and for $j=2,3\ldots$,
\begin{equation}\label{eq:estimate for N11 norm of eta j prime}
\int_{X} g_{\eta_j'}\,d\mu
\le \int_{X} g_{\eta_j}\,d\mu +\int_{X} g_{\rho_j}\,d\mu
< 2^{-j-1}\eps+2^{-j-1}\eps=2^{-j}\eps.
\end{equation}
Also,
\[
\Vert \eta_j'\Vert_{L^1(X)}\le \Vert \eta_j\Vert_{L^1(X)}< 2^{-j}\eps
\]
for all $j\in\N$.
Then $\eta_j'\in N_0^{1,1}(U)$ for all $j\in\N$.
Let $\eta:=\sup_{j\in\N}\eta_j'$.
By \eqref{eq:eta prime one} we have $\eta=1$ in the
set $\bigcup_{j=1}^{\infty}(W_j\setminus \overline{G_j'}^1)$,
which is $1$-finely open and
contains $1$-quasi all of $F$ since we had $\capa_1(\overline{G_j'}^1)\to 0$.
Then by \eqref{eq:quasieverywhere equivalence class} we can redefine $\eta=1$ in $F$; by Theorem \ref{thm:finely open is quasiopen and vice versa}
we now have that $\eta=1$ in a $1$-quasiopen
set containing $F$.
Moreover, by Lemma \ref{lem:sup is upper gradient},
\eqref{eq:estimate for N11 norm of eta 1 prime}, and
\eqref{eq:estimate for N11 norm of eta j prime} we find that
\[
\int_X g_{\eta}\,d\mu
\le \sum_{j=1}^{\infty}\int_X g_{\eta_j'}\,d\mu
< C_d \mathcal H(F)+\eps,
\]
and we also have
\[
\int_X \eta\,d\mu\le \sum_{j=1}^{\infty}\int_X \eta_j'\,d\mu< \eps.
\]
Thus
$\eta\in N^{1,1}(X)$. Clearly also
$\eta=0$ in $X\setminus U$, so that $\eta\in N_0^{1,1}(U)$.

Finally we show that $\eta=0$ in a $1$-quasiopen set containing $X\setminus U$.
Fix $\delta>0$.
From Lemma \ref{lem:Sobolev functions for finite Hausdorff measure 1} we
had that for every $j\in\N$,
\[
\capa_1(\overline{\{\eta_j'>0\}}^1\setminus (U\cup G_j))
\le\capa_1(\overline{\{\eta_j>0\}}^1\setminus (U\cup G_j))=0,
\]
and then since $\eta_j'=0$ in the open set $G_j$,
\begin{equation}\label{eq:capacity outside U is zero}
\capa_1(\overline{\{\eta_j'>0\}}^1\setminus U)=0.
\end{equation}
For $N=2,3,\ldots$ we have by \eqref{eq:capacity of eta j set}
\[
\sum_{j=N}^{\infty}\capa_1\left(\{\eta'_j>0\}\right)
\le  \sum_{j=N}^{\infty}\capa_1\left(\{\eta_j>0\}\right)
\le  \sum_{j=N}^{\infty} 2^{-j}\eps=2^{-N+1}\eps<\delta
\]
for large enough $N$.
Then by \eqref{eq:capacity of fine closure},
also $\capa_1\left(\overline{\bigcup_{j=N}^{\infty}\{\eta'_j>0\}}^1\right)<\delta$.
Now by the characterization \eqref{eq:characterization of fine closure}, we see that
\[
\overline{\{\eta>0\}}^1
\subset \bigcup_{j=1}^{N-1} \overline{\{\eta_j'>0\}}^1\cup 
\overline{\bigcup_{j=N}^{\infty}\{\eta_j'>0\}}^1,
\]
and so by \eqref{eq:capacity outside U is zero},
\[
\capa_1\left(\overline{\{\eta>0\}}^1\setminus U\right)
\le \capa_1\left(\overline{\bigcup_{j=N}^{\infty}\{\eta'_j>0\}}^1\right)<\delta.
\]
Since $\delta>0$ was arbitrary, we have
$\capa_1\left(\overline{\{\eta>0\}}^1\setminus U\right)=0$.
The set
$X\setminus \overline{\{\eta>0\}}^1$ is $1$-finely open, and
then by Theorem \ref{thm:finely open is quasiopen and vice versa},
$X\setminus (\overline{\{\eta>0\}}^1\cap U)$
is $1$-quasiopen. Thus $\eta=0$ in a $1$-quasiopen set containing $X\setminus U$.
\end{proof}

Now we wish to show, essentially, that the converse to Proposition
\ref{prop:one direction} holds when
$A$ is a set of finite perimeter. Note that a set of finite perimeter can
be perturbed in any set of $\mu$-measure zero without changing the perimeter.
However, a set of $\mu$-measure zero may well have an effect on
Newton-Sobolev norms; in Example \ref{ex:non strict subset} we have
$\mathcal L^2(K)=0$ and so $P(K,\R^2)=0$, but
$K$ was not a $1$-strict subset of $U$.

For this reason, we always need to consider a reasonable representative
of a set of finite perimeter $E$.
We choose this representative to be the measure-theoretic interior $I_E$,
as defined in \eqref{eq:measure theoretic interior};
note that by Lebesgue's differentiation theorem,
we indeed have $\mu(I_E\Delta E)=0$.

The proof of the following lemma can be found e.g. in \cite[Lemma 1.34]{BB}.

\begin{lemma}\label{lem:subcurves}
If $\Gamma$ and $\Gamma'$ are families of curves such that for every
$\gamma\in \Gamma$ there exists a subcurve $\gamma'\in\Gamma'$ of $\gamma$,
then $\Mod_1(\Gamma)\le \Mod_1(\Gamma')$.
\end{lemma}

Recall the definition of the Dirichlet spaces $D^1(\cdot)$ from page
\pageref{definition of Dirichlet spaces}.

\begin{theorem}\label{thm:strict subset theorem}
Let $D, U\subset X$ with $U$ $1$-quasiopen
and let $E\subset X$ be $\mu$-measurable with
$\mathcal H(\partial^*E\cap U)<\infty$.
Suppose that also
\[
\capa_1(U\cap (I_E\cup\partial^*E)\setminus \fint D)=0.
\]
Let $\eps>0$.
Then there exists $\rho\in D_0^{1}(D,U)$ such that $\rho=1$
$1$-q.e. in $(I_E\cup \partial^*E)\cap U$,
$\Vert \rho-\ch_E\Vert_{L^1(U)}<\eps$, and
\[
\int_U g_{\rho}\,d\mu< C_d\mathcal H(\partial^*E\cap U)+\eps.
\]
\end{theorem}

Note that the condition $\mathcal H(\partial^*E\cap U)<\infty$
is satisfied by any set of finite perimeter $E$, more precisely if
$P(E,\Om)<\infty$ for some open set $\Om\supset U$.
Note also that if $\ch_E\in L^1(U)$, then $\rho\in N^{1,1}_0(D,U)$.

\begin{proof}
The set $\fint D$ is $1$-quasiopen by Theorem
\ref{thm:finely open is quasiopen and vice versa}.
By Proposition \ref{prop:Sobolev functions for finite Hausdorff measure 2}
we find a function $\eta\in N_0^{1,1}(\fint D)\subset N_0^{1,1}(\fint D,U)$ such that
$\eta=1$ in $\partial^*E\cap \fint D\cap U$, $\Vert \eta\Vert_{L^1(X)}<\eps$,
and
\begin{equation}\label{eq:g eta estimate}
\int_X g_{\eta}\,d\mu< C_d\mathcal H(\partial^*E\cap \fint D\cap U)+\eps.
\end{equation}
Define
\[
\rho:=
\begin{cases}
\eta & \textrm{in }U\setminus (D\cap I_E),\\
1 & \textrm{in }U\cap D\cap I_E.
\end{cases}
\]
Since
$\capa_1(U\cap (I_E\cup\partial^*E)\setminus \fint D)=0$, we have
$\rho=1$
$1$-q.e. in $(I_E\cup \partial^*E)\cap U$, as desired.
Also,
\[
\Vert \rho-\ch_E\Vert_{L^1(U)}\le \Vert \eta\Vert_{L^1(U)}<\eps
\]
as desired.
Now we show that in the set $U$ we have $g_{\rho}\le g_{\eta}$,
where $g_{\rho}$ is the minimal $1$-weak upper gradient of $\rho$ in $U$.
Choose a curve $\gamma$ in $U$. If $\gamma$ lies entirely in $U\setminus (D\cap I_E)$,
then $\rho=\eta$ on this curve and so the pair $(\rho,g_{\eta})$
satisfies the upper gradient inequality
on $1$-a.e. such curve $\gamma$. If $\gamma$ lies entirely in $D\cap I_E$, then
$\rho=1$ on the curve and so again the upper gradient inequality is satisfied.

Assume then that $\gamma$ intersects both $D\cap I_E$ and $U\setminus (D\cap I_E)$;
by splitting $\gamma$ into two subcurves and reversing direction,
if necessary,
we can assume that
$\gamma(0)\in D\cap I_E$ and $\gamma(\ell_{\gamma})\in U\setminus (D\cap I_E)$.
Since we had $\capa_1(U\cap (I_E\cup \partial^*E)\setminus \fint D)=0$,
by \cite[Proposition 1.48]{BB} we know that $1$-a.e. curve avoids
$U\cap (I_E\cup \partial^*E)\setminus \fint D$.
Thus we can assume that $\gamma(\ell_{\gamma})\in U\setminus I_E$, and
then by Proposition \ref{prop:ae curve goes through boundary quasiopen case}
we can assume that there is $t\in (0,\ell_{\gamma}]$
such that $\gamma(t)\in\partial^*E\cap \fint D$; note that here we use also
Lemma \ref{lem:subcurves}.
We can also assume that the pair $(\eta,g_{\eta})$ satisfies
the upper gradient inequality on $\gamma$.
Then
\[
|\rho(0)-\rho(\ell_{\gamma})|=|1-\eta(\ell_{\gamma})|=|\eta(t)-\eta(\ell_{\gamma})|
\le\int_{\gamma}g_{\eta}\,ds.
\]
In total, we have established that $g_{\rho}\le g_{\eta}$ in $U$.
Thus by \eqref{eq:g eta estimate},
\[
\int_U g_{\rho}\,d\mu< C_d\mathcal H(\partial^*E\cap U)+\eps,
\]
as desired. Now $\rho\in D_0^1(D,U)$.
\end{proof}

\section{Applications in the study of capacities}\label{sec:applications}

In this section we apply the results of the previous section
to the study of variational
capacities. We begin with the proof of the first theorem
in the introduction.

\begin{proof}[Proof of Theorem \ref{thm:strict subset theorem intro}]
	To prove the ``only if'' direction,
	assume that $\rcapa_1(I_E,D)<\infty$.
	Thus there exists $u\in N_0^{1,1}(D)$ with $u=1$ in $I_E$.	
	By applying Proposition \ref{prop:one direction}
	with the choices $A=I_E$ and $U=X$, we obtain that
	$\capa_1(\overline{I_E}^1\setminus \fint D)=0$.
	
	To prove the ``if'' direction, let $\eps>0$. We note that
	$\mathcal H(\partial^*E)<\infty$ by \eqref{eq:consequence theta},
	and $\capa_1((I_E\cup\partial^*E)\setminus \fint D)=0$
	by \eqref{eq:meas th closure belongs to base}. Thus
	by Theorem \ref{thm:strict subset theorem}
	we find a function $\rho\in D_0^{1}(D)$ such that $\rho=1$
	$1$-q.e. in $I_E\cup \partial^*E$, $\Vert \rho-\ch_E\Vert_{L^1(X)}<\eps$, and
	\[
	\int_X g_{\rho}\,d\mu< C_d\mathcal H(\partial^*E)+\eps
	\le C_d\alpha^{-1}P(E,X)+\eps,
	\]
	using also \eqref{eq:consequence theta}.
	Since we assume $E$ to be bounded,
	$\rho\in L^1(X)$ and so in fact $\rho\in N_0^{1,1}(D)$.
	Hence
	\[
	\rcapa_1(I_E,D)\le \int_X g_{\rho}\,d\mu\le C_d\alpha^{-1}P(E,X)+\eps,
	\]
	so that letting $\eps\to 0$ we get the conclusion.
\end{proof}

Now we define the variational $1$-capacity in more general
(ambient) sets than the entire space $X$.

\begin{definition}\label{def:relative one capacity}
	Let $A\subset D$ and let $U\subset X$ be $\mu$-measurable. We define
	\[
	\rcapa_1(A,D,U):=\inf \int_U g_u\,d\mu,
	\]
	where the infimum is taken over functions $u\in N^{1,1}_0(D,U)$ such that
	$u=1$ in $A\cap U$,
	and $g_u$ is the minimal $1$-weak upper gradient of $u$ in $U$.
\end{definition}

Sometimes $(A,D^c,U)$ is called a condenser and
$\rcapa_{1}(A,D,U)$ is called the capacity of the condenser, see e.g.
\cite{HaSh}, as well as \cite{HK,KK,Maz} where $\rcapa_p$ for more general
$p\ge 1$ is considered.
Note that by \eqref{eq:quasieverywhere equivalence class}
we can equivalently require that $u=1$ $1$-q.e. in $A\cap U$.

Now we can show that the capacity of a condenser that consists of two sets of finite perimeter
is finite if and only if the sets do not ``touch'' each other.

\begin{theorem}\label{thm:condenser}
Let $\Om\subset X$ be open and bounded and let $E,F\subset \Om$ with $P(E,\Om)<\infty$,
$P(F,\Om)<\infty$, and $E\cap F=\emptyset$. Then
$\rcapa_1(I_E,I_F^c,\Om)<\infty$ if and only if
$\mathcal H(\partial^*E\cap \partial^*F\cap\Om)=0$.
Moreover, then $\rcapa_1(I_E,I_F^c,\Om)\le C\min\{P(E,\Om),P(F,\Om)\}$
for a constant $C=C(C_d,C_P,\lambda)$.
\end{theorem}

\begin{proof}
The set $\Om\cap O_F$ is $1$-quasiopen by
Proposition \ref{prop:set of finite perimeter is quasiopen},
and thus by Theorem \ref{thm:finely open is quasiopen and vice versa},
\begin{equation}\label{eq:OF and fint OF}
\mathcal H(\Om\cap O_F\setminus \fint O_F)=0.
\end{equation}
For sets $A_1,A_2\subset X$, we write $A_1\overset{\mathcal H-a.e.}{=}A_2$
if $\mathcal H(A_1\Delta A_2)=0$.
Now we get
\begin{align*}
\Om\cap (I_E\cup \partial^*E)\setminus \fint I_F^c
&=
\Om\cap (I_E\cup \partial^*E)\setminus \fint O_F\quad\textrm{by Lemma }\ref{lem:fint of IE c}\\
&\overset{\mathcal H-a.e.}{=} 
\Om\cap (I_E\cup \partial^*E)\setminus O_F\quad\textrm{by }\eqref{eq:OF and fint OF}\\
&=\Om\cap (I_E\cup \partial^*E)\cap (I_F\cup \partial^*F)\\
&=\Om\cap \partial^*E\cap \partial^*F\quad\textrm{since }E\cap F=\emptyset.
\end{align*}
If $\rcapa_1(I_E,I_F^c,\Om)<\infty$, then by Proposition
\ref{prop:one direction} and \eqref{eq:null sets of Hausdorff measure and capacity}
we know that
$\mathcal H(\Om\cap \overline{I_E}^1\setminus \fint I_F^c)=0$.
Then also $\mathcal H(\Om\cap (I_E\cup \partial^*E)\setminus \fint I_F^c)=0$ by \eqref{eq:meas th closure belongs to base}, and so
$\mathcal H(\Om\cap \partial^*E\cap \partial^*F)=0$.

Conversely, if $\mathcal H(\partial^*E\cap \partial^*F\cap\Om)=0$,
then $\mathcal H(\Om\cap (I_E\cup\partial^*E)\setminus \fint I_F^c)=0$
and so $\capa_1(\Om\cap (I_E\cup\partial^*E)\setminus \fint I_F^c)=0$
by \eqref{eq:null sets of Hausdorff measure and capacity}. Let $\eps>0$.
Since we also have $\mathcal H(\partial^*E\cap \Om)<\infty$ by
\eqref{eq:consequence theta}, 
we can apply Theorem \ref{thm:strict subset theorem} with the sets
$E$, $D=I^c_F$, and $U=\Om$, to find a function $\rho\in N_0^{1,1}(I^c_F,\Om)$
such that $\rho=1$ $1$-q.e. in $(I_E\cup \partial^*E)\cap \Om$
and
\[
\int_\Om g_{\rho}\,d\mu< C_d\mathcal H(\partial^*E\cap \Om)+\eps.
\]
Since $\eps>0$ was arbitrary, using also \eqref{eq:consequence theta} we now obtain
\[
\rcapa_{1}(I_E, I_F^c,\Om)\le C_d\mathcal H(\partial^*E\cap \Om)\le \alpha^{-1}C_d P(E,\Om).
\]
By the exactly analogous reasoning, we get
$\rcapa_{1}(I_F, I_E^c,\Om)\le \alpha^{-1}C_d P(F,\Om)$,
and since clearly $\rcapa_{1}(I_F, I_E^c,\Om)=\rcapa_{1}(I_E, I_F^c,\Om)$,
we get the conclusion
\[
\rcapa_1(I_E,I_F^c,\Om)\le C\min\{P(E,\Om),P(F,\Om)\}<\infty
\]
with $C=\alpha^{-1}C_d$.
\end{proof}

It is perhaps interesting that the quantity $\rcapa_1(I_E,I_F^c,\Om)$
can never take a large finite value; it is either at most of the order 
$\min\{P(E,\Om),P(F,\Om)\}$, or else it is infinite.
The analogous $p$-capacity for $p>1$ typically becomes arbitrarily large
as the sets $E$ and $F$ get closer to each other.

Now we define two different $\BV$-versions of the variational $1$-capacity.
Recall the definitions of the approximate limits $u^{\wedge}$ and $u^{\vee}$
from \eqref{eq:lower approximate limit} and \eqref{eq:upper approximate limit};
by Lebesgue's differentiation theorem, $u=u^{\wedge}=u^{\vee}$
a.e. for any locally integrable function $u$.
In the case $u=\ch_E$ with $E\subset X$, we have $x\in I_E$ if and only if $u^{\wedge}(x)=u^{\vee}(x)=1$, $x\in O_E$ if and only if $u^{\wedge}(x)=u^{\vee}(x)=0$, and $x\in \partial^*E$ if and only if $u^{\wedge}(x)=0$ and $u^{\vee}(x)=1$.

\begin{definition}\label{def:variational capacities}
Let $A\subset D$ and let $U\subset X$ be $\mu$-measurable.
We define the variational $\BV$-capacity by
	\[
	\rcapa_{\BV}(A,D,U):=\inf \Vert Du\Vert(U),
	\]
	where the infimum is taken over functions $u\in L^1(U)$ such that
	$u^{\wedge}=u^{\vee}= 0$ $\mathcal H$-a.e. in $U\setminus D$ and
	$u^{\wedge}\ge 1$ $\mathcal H$-a.e. in $A\cap U$.
	
	We define an alternative version of the variational $\BV$-capacity by
	\[
	\rcapa^{\vee}_{\BV}(A,D,U):=\inf \Vert Du\Vert(U),
	\]
	where the infimum is taken over functions $u\in L^1(U)$ such that
	$u^{\wedge}=u^{\vee}= 0$ $\mathcal H$-a.e. in $U\setminus D$ and
	$u^{\vee}\ge 1$ $\mathcal H$-a.e. in $A\cap U$.
	If $U=X$, we omit it from the notation.
\end{definition}

By truncation we see that it is enough to consider test functions $0\le u\le 1$.
Note that the condition $\Vert Du\Vert(U)<\infty$ implicitly means
that $\Vert Du\Vert(\Om)<\infty$ for some open
$\Om\supset U$.
It is obvious that always $\rcapa^{\vee}_{\BV}(A,D,U)\le \rcapa_{\BV}(A,D,U)$,
and in \cite[Eq. (4.2)]{L-Appr} it was noted that also
\begin{equation}\label{eq:BV and one capacity comparison}
\rcapa_{\BV}(A,D,U)\le \rcapa_{1}(A,D,U)
\end{equation}
whenever $U$ is open.

In \cite{CDLP} it was shown, with rather different methods compared to ours
and with slightly different definitions,
that in Euclidean spaces one has
$\rcapa^{\vee}_{\BV}(A,\R^n)=\rcapa_{\BV}(A,\R^n)$ for every $A\subset \R^n$.
A definition similar to $\rcapa^{\vee}_{\BV}(A,D)$
was also studied (in metric spaces) in \cite{HaSh},
in the case where $A$ is a compact subset of an open set $D$;
it follows from \cite[Theorem 4.5, Theorem 4.6]{HaSh} that
\begin{equation}\label{eq:Ha Sh capacity result}
\rcapa_{\BV}(A,D)\le C\rcapa^{\vee}_{\BV}(A,D).
\end{equation}
The constant $C\ge 1$ is indeed necessary in the metric space setting,
see \cite[Example 4.4]{HaSh}.
It is of interest to study capacities for more general sets;
in \cite{L-Appr} it was shown that
\[
\rcapa_{\BV}(A,D)=\rcapa_1(A,D)
\]
when $A\subset D$ and both $D$ and $X\setminus A$
are $1$-quasiopen, and this was used to prove an approximation result for
$\BV$ functions.
Now we wish to complement these results with the following theorem on the capacity
$\rcapa^{\vee}_{\BV}$; note that $A$ is now a completely general set and so the theorem
greatly strengthens \eqref{eq:Ha Sh capacity result}.


\begin{theorem}
Let $A\subset V$ and $U\subset X$ such that $V$ and $U$ are $1$-quasiopen.
Then
\[
\rcapa_{1}(A,V,U)\le C\rcapa^{\vee}_{\BV}(A,V,U)
\]
for a constant $C=C(C_d,C_P,\lambda)\ge 1$.
In particular, if $U$ is open,
\[
\rcapa_{\BV}(A,V,U)\le C\rcapa^{\vee}_{\BV}(A,V,U).
\]
\end{theorem}
\begin{proof}
We can assume that $\rcapa^{\vee}_{\BV}(A,V,U)<\infty$. Let $\eps>0$.
Pick an open set $U_e\supset U$ (we can assume that $\mu(U_e\setminus U)<\infty$) and a function $u\in L^1(U_e)$
such that $0\le u\le 1$ in $U_e$,
 $u^{\wedge}=u^{\vee}= 0$ $\mathcal H$-a.e. in $U\setminus V$,
	$u^{\vee}= 1$ $\mathcal H$-a.e. in $A\cap U$, and
	\[
	\Vert Du\Vert(U_e)<\rcapa^{\vee}_{\BV}(A,V,U)+\eps.
	\]
	By the coarea formula \eqref{eq:coarea}, we find
	$t\in (0,1)$ such that $E:=\{u>t\}$ satisfies
	$P(E,U_e)\le\Vert Du\Vert(U_e)$.
	Then
	\begin{equation}\label{eq:A IE partial E}
	\mathcal H(A\cap U\setminus (I_E\cup\partial^*E))
	\le  \mathcal H(A\cap U\setminus \{u^{\vee}=1\})=0
	\end{equation}
	and similarly
	\[
	\mathcal H(U\cap (I_E\cup \partial^*E)\setminus  V)
	\le \mathcal H(U\cap \{u^{\vee}>0\}\setminus  V)=0.
	\]
	Now by Theorem \ref{thm:finely open is quasiopen and vice versa} and \eqref{eq:null sets of Hausdorff measure and capacity},
	\[
	\mathcal H(U\cap (I_E\cup \partial^*E)\setminus  \fint V)
	=\mathcal H(U\cap (I_E\cup \partial^*E)\setminus  V)=0.
	\]
	Moreover, by \eqref{eq:consequence theta},
	$\mathcal H(\partial^*E\cap U_e)<\infty$.
	Now by Theorem \ref{thm:strict subset theorem}
	we find a function $\rho\in D_0^{1}(V,U)$ such that
	$\rho=1$ $1$-q.e. in $(I_E\cup\partial^*E)\cap U$,
	$\Vert \rho-\ch_E\Vert_{L^1(U)}<\eps$ (and then in fact $\rho\in N_0^{1,1}(V,U)$),
	and $\int_U g_{\rho}\,d\mu<C_d\mathcal H(\partial^*E\cap U)+\eps$.
	Then by \eqref{eq:consequence theta},
	\[
	\int_U g_{\rho}\,d\mu<C_d\mathcal H(\partial^*E\cap U)+\eps
	\le C_d\mathcal H(\partial^*E\cap U_e)+\eps
	\le C_d\alpha^{-1} P(E,U_e)+\eps.
	\]
	Since $\rho=1$ $1$-q.e. in $(I_E\cup\partial^*E)\cap U$,
	also $\rho=1$ $1$-q.e. in $A\cap U$ by \eqref{eq:A IE partial E}
	and \eqref{eq:null sets of Hausdorff measure and capacity}.
Thus
	\begin{align*}
	\rcapa_1(A,V,U)\le \int_U g_{\rho}\,d\mu
	&< C_d\alpha^{-1} P(E,U_e)+\eps\\
	&< C_d\alpha^{-1} (\rcapa_{\BV}^{\vee}(A,V,U)+\eps)+\eps.
	\end{align*}
	Letting $\eps\to 0$, we get the first claim.
	The second claim then follows from \eqref{eq:BV and one capacity comparison}.
\end{proof}

Even though $A$ is allowed to be an arbitrary set,
the assumption that $V$ is $1$-quasiopen cannot be removed,
as demonstrated by the following example.

\begin{example}
Let $X=\R$ (unweighted) and let $A=V=[0,1]$ and $U=\R$. Then
\[
\rcapa^{\vee}_{\BV}(A,V,U)\le \Vert D\ch_{[0,1]}\Vert(\R)=2,
\]
but $\rcapa_{\BV}(A,V,U)=\infty$ since there are no admissible functions.
\end{example}

\section{An approximation result for BV functions}\label{sec:approximation}

In this section we apply our theory of $1$-strict subsets to prove
a pointwise approximation result for $\BV$ functions, given in Theorem
\ref{thm:approximation from above intro} in the introduction.

\begin{lemma}\label{lem:disjoint quasiopen sets}
Let $S_1,\ldots ,S_n\subset X$ be pairwise disjoint Borel sets that are of finite
$\mathcal H$-measure. Then there exist pairwise disjoint $1$-quasiopen sets $U_j\supset S_j$,
$j=1,\ldots,n$.
\end{lemma}
\begin{proof}
By Lemma \ref{lem:set of finite Hausdorff measure quasiclosed} the set
$X\setminus \bigcup_{k=2}^n S_k$ is $1$-quasiopen, and contains $S_1$.
Thus by Proposition \ref{prop:Sobolev functions for finite Hausdorff measure 2}
we find a function $\eta_1\in N^{1,1}_0(X\setminus \bigcup_{k=2}^n S_k)$
with $\eta_1=1$ in $S_1$.
By the quasicontinuity of Newton-Sobolev functions, it is straightforward to check that
$\{\eta_1>1/2\}$ and $X\setminus \{\eta_1\ge 1/2\}$
are $1$-quasiopen sets (see e.g. \cite[Proposition 3.4]{BBM}).
We can do the same for each set $S_1,\ldots, S_n$.
Then define for each $j=1,\ldots,n$
\[
U_j:=\{\eta_j>1/2\}\setminus \bigcup_{\substack{k=1 \\ k\neq j}}^n\{\eta_k\ge 1/2\}.
\]
Now each set $U_j$ contains $S_j$ and is a $1$-quasiopen set by
the fact that every finite intersection of $1$-quasiopen sets is
$1$-quasiopen (see e.g. \cite[Lemma 2.3]{Fug}).
\end{proof}

Next we prove the following Leibniz rule.

\begin{lemma}\label{lem:Leibniz}
Let $\Om\subset X$ be open and let $U_1,\ldots, U_n\subset \Om$ be pairwise disjoint
$1$-quasiopen sets.
For each $j=1,\ldots,n$ let $\eta_j\in N^{1,1}_0(U_j)$
with $0\le \eta_j\le 1$ on $X$,
$\eta_j=0$ in a $1$-quasiopen set containing $X\setminus U_j$,
and such that there is a
$1$-quasiopen set $V_j\subset \{\eta_j=1\}$, $j=1,\ldots,n$.
Let $U_0$ be another $1$-quasiopen set with
$U_0\cup V_1\cup \ldots \cup V_n=\Om$ and let
$\eta:=\sum_{j=1}^n\eta_j$, and finally suppose that
$v\in N^{1,1}(U_0)$ and $\rho_j\in N^{1,1}(U_j)$ for each
$j=1,\ldots,n$.
Then
\[
w:=\sum_{j=1}^n \eta_j \rho_j +(1-\eta)v
\]
has a $1$-weak upper gradient (in $\Om$)
\[
g:= \sum_{j=1}^n \eta_j g_{\rho_j} +(1-\eta)g_v+\sum_{j=1}^n g_{\eta_j}|\rho_j-v|,
\]
where each $g_{\rho_j}$ is the minimal $1$-weak upper gradient of
$\rho_j$ in $U_j$, and $g_v$ is the minimal $1$-weak upper gradient of
$v$ in $U_0$.
\end{lemma}

Note that $g_{\eta_j}=0$ outside $U_j\cap U_0$ by \eqref{eq:upper gradient in coincidence set}, and so the function $g$ is well defined in the whole of $\Om$.

\begin{proof}
	Using the fact that $\eta_j=0$ in a $1$-quasiopen set containing $X\setminus U_j$, and
	the fact that finite intersections of $1$-quasiopen
	sets are $1$-quasiopen, we find a $1$-quasiopen set
	$V\subset \Om$ containing
	$\Om\setminus \bigcup_{j=1}^n U_j$ but not intersecting any of the sets
	$\{\eta_j>0\}$.
By \cite[Remark 3.5]{S2} we know that $V$ is $1$-path open, meaning that for
$1$-a.e. curve $\gamma$, the set $\gamma^{-1}(V)$ is a relatively open
subset of $[0,\ell_{\gamma}]$.
The same holds for each of the sets $U_0\cap U_j$ and $V_j$.
Let $\gamma$ be a curve such that this property for preimages
holds for all subcurves of $\gamma$.
In the set $V$ we know that $g_v= g$ is $1$-weak upper gradient of $v=w$.
In each $V_j$, the function $g_{\rho_j}= g$ is a $1$-weak upper
gradient of $\rho_j=w$.
Finally, by the Leibniz rule given in \cite[Lemma 2.18]{BB}, in each set $U_j\cap U_0$
the function 
\[
\eta_j g_{\rho_j} +(1-\eta_j)g_v+ g_{\eta_j}|\rho_j-v|
=\eta_j g_{\rho_j} +(1-\eta)g_v+ g_{\eta_j}|\rho_j-v|= g
\]
is a $1$-weak upper gradient of $\eta_j u+(1-\eta_j)v=w$.
Assume further that the pair $(w,g)$ satisfies the upper gradient
inequality on each subcurve of $\gamma$ lying either in $V$,
in one of the sets $U_j\cap U_0$, or in one of the sets
$V_j$. By Lemma \ref{lem:subcurves}
these properties are satisfied by $1$-a.e. curve $\gamma$.

Note that we can write the entire $\Om$ as the union of $1$-quasiopen sets
\[
\Om=V\cup \bigcup_{j=1}^n U_j
=V\cup \bigcup_{j=1}^n\left[(U_0\cap U_j)\cup V_j\right].
\]
Since $[0,\ell_{\gamma}]$ is a compact set,
the curve $\gamma$ can be broken into a finite number
of subcurves each of which lies either in $V$,
or in one of the sets $U_0\cap U_j$, or in $V_j$.
Summing up over the subcurves, we find that the pair $(w,g)$ satisfies the upper gradient
inequality on $\gamma$, and thus $g$ is a $1$-weak upper gradient of $w$ in $\Om$.
\end{proof}

\begin{proposition}[{\cite[Proposition 3.6]{L-DC}}]\label{prop:approximation with L infinity bound}
	Let $U\subset \Om\subset X$ be such that $U$ is $1$-quasiopen and
	$\Om$ is open, and let $u\in\BV(\Om)$ and $\beta>0$ such that
	$u^{\vee}-u^{\wedge}<\beta$ in $U$.
	Then there exists a sequence
	$(u_i)\subset N^{1,1}(U)$ such that
	$u_i\to u$ in $L^1(U)$, $\sup_{U}| u_i-u^{\vee}|\le 9\beta$
	for all $i\in\N$, and
	\[
	\lim_{i\to\infty}\int_{U} g_{u_i}\,d\mu=\Vert Du\Vert(U),
	\]
	where each $g_{u_i}$
	is the minimal $1$-weak upper gradient of $u_i$ in $U$.
\end{proposition}

\begin{proof}[Proof of Theorem \ref{thm:approximation from above intro}]
We begin by decomposing the jump set $S_u=\{u^{\wedge}<u^{\vee}\}$
into the pairwise disjoint sets
$S_1:=\{u^{\vee}-u^{\wedge}\ge 1\}$ and
$S_j:=\{1/j\le u^{\vee}-u^{\wedge}< 1/(j-1)\}$
for $j=2,3,\ldots$ (all understood to be subsets of $\Om$).
By the decomposition \eqref{eq:variation measure decomposition}, we have
$\mathcal H(S_j)<\infty$ for every $j\in\N$.
Applying Lemma \ref{lem:using quasi Lindelof}, for each $j\in\N$ we find
open sets $W_{j1}\supset W_{j2}\supset \ldots\supset S_j$ such that
\begin{equation}\label{eq:Wji and Sj}
\capa_1\left(\bigcap_{i=1}^{\infty}W_{ji}\setminus S_j\right)=0.
\end{equation}
We can also assume that these are subsets of $\Om$ and that
$\Vert Du\Vert(W_{ji})<\Vert Du\Vert(S_j)+1/i^2$.
By Lemma \ref{lem:disjoint quasiopen sets}, for each $i\in\N$
and $j\le i$
we find $1$-quasiopen sets $V_{ji}\supset S_j$ such that
$V_{1i},\ldots V_{ii}$ are pairwise disjoint.
Moreover, note that the sets
$\{u^{\vee}-u^{\wedge}<1/(j-1)\}$
are $1$-quasiopen
by Proposition \ref{prop:quasisemicontinuity}.
Thus the sets $U_{ji}:=W_{ji}\cap V_{ji}\cap \{u^{\vee}-u^{\wedge}<1/(j-1)\}\supset S_j$ are
$1$-quasiopen.
Now for each $i\in\N$, $U_{1i},\ldots,U_{ii}$ are
pairwise disjoint sets with
\begin{equation}\label{eq:Uji and Sj}
\Vert Du\Vert(U_{ji})<\Vert Du\Vert(S_j)+\frac{1}{i^2}
\end{equation}
for all $j=1,\ldots,i$.

Using Proposition \ref{prop:Sobolev functions for finite Hausdorff measure 2},
take functions $\eta_{ji}\in N_0^{1,1}(U_{ji})$ such that $0\le \eta_{ji}\le 1$ on $X$,
$\eta_{ji}=1$ in a $1$-quasiopen set containing $S_j$,
and $\eta_{ji}=0$ in a $1$-quasiopen set containing $X\setminus U_{ji}$.

Let $u_i:=\max\{-i,u\}$ for each $i\in\N$.
By Lemma \ref{lem:set of finite Hausdorff measure quasiclosed},
each $\Om\setminus (S_1\cup \ldots \cup S_i)$ is a $1$-quasiopen set
in which $u_i^{\vee}-u_i^{\wedge}\le u^{\vee}-u^{\wedge}<1/i$.
Thus by Proposition \ref{prop:approximation with L infinity bound}
we find a function $v_{i}\in N^{1,1}(\Om\setminus (S_1\cup \ldots \cup S_i))$
such that
\begin{equation}\label{eq:vi minus u L1}
|v_i- u_i^{\vee}|\le 9/i\ \ \textrm{in }\Om\setminus (S_1\cup \ldots \cup S_i)\quad\textrm{and}\quad
\Vert v_i-u_i\Vert_{L^1(\Om\setminus (S_1\cup \ldots \cup S_i))}<1/i
\end{equation}
and
\begin{equation}\label{eq:gvi in Om setminus S1 Sn}
\int_{\Om\setminus (S_1\cup \ldots \cup S_i)}g_{v_{i}}\,d\mu
<\Vert Du\Vert(\Om\setminus (S_1\cup \ldots \cup S_i))+\frac{1}{i},
\end{equation}
where $g_{v_i}$ is the minimal $1$-weak upper gradient
of $v_i$ in $\Om\setminus (S_1\cup \ldots \cup S_i)$.
It is easy to show that $\mu(S_1\cup \ldots \cup S_n)=0$
since the $S_j$'s are sets of finite $\mathcal H$-measure
(see e.g. \cite[Lemma 6.1]{KKST}).
Since $L^1$-convergence implies pointwise convergence a.e. for
a subsequence,
by Lebesgue's dominated convergence theorem
(using \eqref{eq:vi minus u L1}) 
we can also assume that
\begin{equation}\label{eq:u minus vi condition}
\int_{\Om}g_{\eta_{ji}}|v_i-u_i|\,d\mu<\frac{1}{i^2}
\end{equation}
for each $j=1,\ldots,i$.
Also, for each $i\in\N$ let
\[
\alpha_i:=\frac{1}{i^2}\min \left\{1, \Big(\int_{X}g_{\eta_{1i}}\,d\mu\Big) ^{-1},\ldots,
\Big(\int_{X}g_{\eta_{ii}}\,d\mu\Big) ^{-1}\right\}.
\]
Now fix $i\in\N$ and $j\in\{1,\ldots,i\}$.
For all $k\in\N$ pick
\[
\beta_{jik}\in ((k-1) \alpha_i,k\alpha_i)
\]
such that (in what follows, we work with the function $u_i+i$ since it is nonnegative)
\[
P(\{u_i+i>\beta_{jik}\},\Om)<\infty
\]
and
\begin{equation}\label{eq:choice of beta k}
\alpha_i P(\{u_i+i>\beta_{jik}\},U_{ji})\le 
\int_{(k-1)\alpha_i}^{k\alpha_i}P(\{u_i+i>t\},U_{ji})\,dt;
\end{equation}
note that this choice is possible since $P(\{u_i+i>t\},\Om)<\infty$
for a.e. $t\in\R$ by the coarea formula \eqref{eq:coarea}.
Now we will apply Theorem \ref{thm:strict subset theorem}
with the choices
\[
E=\{u_i+i>\beta_{jik}\},\quad D=X\setminus (S_j\cap \{(u_i+i)^{\vee}<\beta_{jik}\}),
\quad\textrm{and}\quad U= U_{ji}.
\]
Note
that if $x\in I_E\cup \partial^*E$, then $(u_i+i)^{\vee}(x)\ge \beta_{jik}$.
Thus $I_E\cup \partial^*E\subset D$.
Also note that $D$ is $1$-quasiopen
by Lemma \ref{lem:set of finite Hausdorff measure quasiclosed}, and so $\mathcal H(D\setminus \fint D)=0$ by
Theorem \ref{thm:finely open is quasiopen and vice versa}.
Thus
\[
\mathcal H((I_E\cup \partial^*E)\setminus \fint D)=0
\]
as required in Theorem \ref{thm:strict subset theorem}.
Clearly also
\begin{equation}\label{eq:u vee and I}
\{(u_i+i)^{\vee}> \beta_{jik}\}\subset I_{\{u_i+i>\beta_{jik}\}}\cup \partial^*\{u_i+i>\beta_{jik}\}.
\end{equation}
Now Theorem \ref{thm:strict subset theorem} gives functions
\[
\rho_{jik}\in N_0^{1,1}(X\setminus (S_j\cap \{(u_i+i)^{\vee}< \beta_{jik}\}),U_{ji})
\]
with $0\le \rho_{jik}\le 1$ in $U_{ji}$, $\rho_{jik}=1$ $1$-q.e.
in $(I_{\{u_i+i>\beta_{jik}\}}\cup \partial^*\{u_i+i>\beta_{jik}\})\cap U_{ji}$
(and by redefining, we can leave out the ``$1$-q.e.'')
and thus in $\{(u_i+i)^{\vee}>\beta_{jik}\}\cap U_{ji}$ by \eqref{eq:u vee and I},
\begin{equation}\label{eq:rho jik L1 estimate}
\Vert\rho_{jik}-\ch_{\{u_i+i>\beta_{jik}\}}\Vert_{L^1(U_{ji})}
<\frac{2^{-k}}{i^2},
\end{equation}
and
\begin{equation}\label{eq:estimate of g rho jik}
\int_{U_{ji}} g_{\rho_{jik}}\,d\mu
< C_d\mathcal H(\partial^*\{u_i+i>\beta_{jik}\}\cap U_{ji})+\frac{2^{-k}}{i^2}.
\end{equation}
Since we can choose the norm
$\Vert \rho_{ijk}-\ch_{\{u_i+i>\beta_{jik}\}}\Vert_{L^1(U_{ji})}$ to be as small as we
like and since $L^1$-convergence implies pointwise convergence for a subsequence,
by Lebesgue's dominated convergence theorem we can also assume that
\begin{equation}\label{eq:g eta ji estimate}
\int_{U_{ji}}g_{\eta_{ji}}|\rho_{jik}-\ch_{\{u_i+i>\beta_{jik}\}}|\,d\mu
<\frac{2^{-k}}{i^2}.
\end{equation}
Then define two functions in the set $U_{ji}$
(both understood to be pointwise defined)
\[
\widetilde{\rho}_{ji}:=\alpha_i\sum_{k=1}^{\infty}\rho_{jik}-i
\quad\textrm{and also}\quad
\widehat{u}_{ji}:=\alpha_i\sum_{k=1}^{\infty}\ch_{\{(u_i+i)^{\vee}>\beta_{jik}\}}-i.
\]
Note that $\widetilde{\rho}_{ji}\ge \widehat{u}_{ji}$  and
$\sup_{U_{ji}}|\widehat{u}_{ji}-u_i^{\vee}|\le \alpha_i$,
and so
\begin{equation}\label{eq:rho ji tilde first estimate}
\widetilde{\rho}_{ji}\ge u_i^{\vee}-\alpha_i\quad\textrm{in }U_{ji}.
\end{equation}
Moreover, since $\rho_{ijk}=0$
in $S_j\cap \{(u_i+i)^{\vee}<\beta_{jik}\}$, it follows that
\begin{equation}\label{eq:rho ji widetilde estimate in Sj}
\widetilde{\rho}_{ji}\le (u_i+i)^{\vee}+\alpha_i-i
=\max\{u^{\vee},-i\}+\alpha_i\quad \textrm{in }S_j.
\end{equation}
Additionally, by \eqref{eq:rho jik L1 estimate} and the fact that $\alpha_i\le 1$,
\[
\Vert \widetilde{\rho}_{ji}-\widehat{u}_{ji}\Vert_{L^1(U_{ji})}\le \sum_{k=1}^{\infty}
\int_{U_{ji}}\alpha_i|\rho_{jik}-\ch_{\{u_i+i>\beta_{jik}\}}|\,d\mu<\frac{1}{i^2}
\]
and so
\begin{equation}\label{eq:rho ji tilde minus ui}
\begin{split}
\Vert \widetilde{\rho}_{ji}-u_i\Vert_{L^1(U_{ji})}
&\le \Vert \widetilde{\rho}_{ji}-\widehat{u}_{ji}\Vert_{L^1(U_{ji})}
+\Vert \widehat{u}_{ji}-u_i\Vert_{L^1(U_{ji})}\\
&<\frac{1}{i^2}+\alpha_i\mu(\Om)
\le \frac{1}{i^2}(1+\mu(\Om))
\end{split}
\end{equation}
(recall that we assume $\mu(\Om)<\infty$).

Using Lemma \ref{lem:sup is upper gradient} we get
(note that $\alpha$ and $\alpha_i\le 1$ denote different quantities)
\begin{align*}
\int_{U_{ji}}g_{\widetilde{\rho}_{ji}}\,d\mu
&\le \alpha_i\sum_{k=1}^{\infty}\int_{U_{ji}}g_{\rho_{jik}}\,d\mu\\
&< \alpha_i\sum_{k=1}^{\infty}\left(C_d\mathcal H(\partial^*\{u_i+i>\beta_{jik}\}\cap U_{ji})
+\frac{2^{-k}}{i^2}\right)\quad\textrm{by }\eqref{eq:estimate of g rho jik}\\
&\le \alpha_i\sum_{k=1}^{\infty}\left(C_d\alpha^{-1}P(\{u_i+i>\beta_{jik}\},
U_{ji})+\frac{2^{-k}}{i^2}\right)\quad\textrm{by }\eqref{eq:consequence theta}\\
&\le \sum_{k=1}^{\infty}\left(C_d\alpha^{-1}\int_{(k-1)\alpha_i}^{k\alpha_i}P(\{u_i+i>t\},U_{ji})\,dt+\frac{2^{-k}}{i^2}\right)\quad\textrm{by }\eqref{eq:choice of beta k}\\
&= C_d\alpha^{-1}\int_{-i}^{\infty}P(\{u_i>t\},U_{ji})\,dt+\frac{1}{i^2}\\
&= C_d \alpha^{-1} \Vert Du_i\Vert(U_{ji})+\frac{1}{i^2}
\end{align*}
by the coarea formula \eqref{eq:coarea}.
Also, by the fact that
$\Vert \widehat{u}_{ji}-u_i\Vert_{L^{\infty}(U_{ji})}\le \alpha_i$
and the choice of $\alpha_i$, we have
\begin{equation}\label{eq:rho ji ui estimate with derivative}
\begin{split}
\int_{U_{ji}}g_{\eta_{ji}}|\widetilde{\rho}_{ji}-u_i|\,d\mu
&\le \int_{U_{ji}}g_{\eta_{ji}}|\widetilde{\rho}_{ji}-\widehat{u}_{ji}|\,d\mu+ 
\int_{U_{ji}}g_{\eta_{ji}}|\widehat{u}_{ji}-u_i|\,d\mu\\
&\le \sum_{k=1}^{\infty}\int_{U_{ji}}g_{\eta_{ji}}|\rho_{jik}
-\ch_{\{u_i+i>\beta_{jik}\}}|\,d\mu+\frac{1}{i^2}\\
&<\frac{2}{i^2}\quad\textrm{by }\eqref{eq:g eta ji estimate}.
\end{split}
\end{equation}
Since $u_i^{\vee}-u_i^{\wedge}\le u^{\vee}-u^{\wedge}<1/(j-1)$ in $U_{ji}$,
by Proposition \ref{prop:approximation with L infinity bound}
we also find a function $v_{ji}\in N^{1,1}(U_{ji})$ such that
\[
u_i^{\vee}\le v_{ji}\le u_i^{\vee}+18/(j-1)\ \textrm{ in } U_{ji}\quad\textrm{and}\quad
\int_{U_{ji}} g_{v_{ji}}\,d\mu<\Vert Du_i\Vert(U_{ji})+\frac{1}{i^2}.
\]
Let $\rho_{ji}:=\min\{\widetilde{\rho}_{ji},v_{ji}\}$.
Then
\begin{equation}\label{eq:estimate of g rho ji}
\int_{U_{ji}}g_{\rho_{ji}}\,d\mu
\le \int_{U_{ji}}g_{\widetilde{\rho}_{ji}}\,d\mu+\int_{U_{ji}}g_{v_{ji}}\,d\mu
< (C_d \alpha^{-1}+1) \Vert Du_i\Vert(U_{ji})+\frac{2}{i^2}.
\end{equation}
Thus
$\rho_{ji}\in N^{1,1}(U_{ji})$.
Also,
\begin{equation}\label{eq:rho ji pointwise bounds}
\rho_{ji}\le u_i^{\vee}+18/(j-1)\quad\textrm{and}\quad\rho_{ji}\ge u_i^{\vee}-\alpha_i
\quad\textrm{in }U_{ji}\quad\textrm{by }\eqref{eq:rho ji tilde first estimate},
\end{equation}
and by \eqref{eq:rho ji tilde minus ui},
\begin{equation}\label{eq:rho ji minus u L1}
\Vert \rho_{ji}-u_i\Vert_{L^1(U_{ji})}\le \Vert \widetilde{\rho}_{ji}-u_i\Vert_{L^1(U_{ji})}
\le \frac{1}{i^2}(1+\mu(\Om)).
\end{equation}
Moreover, by \eqref{eq:rho ji ui estimate with derivative} we have
\[
\int_{U_{ji}}g_{\eta_{ji}}|\rho_{ji}-u_i|\,d\mu
\le \int_{U_{ji}}g_{\eta_{ji}}|\widetilde{\rho}_{ji}-u_i|\,d\mu<\frac{2}{i^2}.
\]
Then using also \eqref{eq:u minus vi condition}, we get
\begin{equation}\label{eq:rho ji minus v i condition}
\int_{U_{ji}}g_{\eta_{ji}}|\rho_{ji}-v_i|\,d\mu<\frac{3}{i^2}.
\end{equation}
Recall that so far we have kept $i\in\N$ and $j\in\{1,\ldots,i\}$ fixed.
Now for each $i\in\N$,
let $\eta_i:=\max_{j\in\{1,\ldots,i\}}\eta_{ji}$.
Define the functions
\[
w_{i}:=\sum_{j=1}^i\eta_{ji}\rho_{ji}+(1-\eta_{i})v_{i}+\frac{9}{i},\quad i\in\N.
\]
Then by \eqref{eq:rho ji minus u L1} and \eqref{eq:vi minus u L1},
\begin{align*}
\Vert w_i-u_i\Vert_{L^1(\Om)}
&\le \sum_{j=1}^{i}\Vert \rho_{ji}-u_i\Vert_{L^1(U_{ji})}
+\Vert v_i-u_i\Vert_{L^1(\Om\setminus (S_1\cup \ldots \cup S_i))}+\frac{9}{i}\mu(\Om)\\
&\le \frac{1}{i}(1+\mu(\Om))+\frac{1}{i}+\frac{9}{i}\mu(\Om)=\frac{2}{i}+\frac{10}{i}\mu(\Om).
\end{align*}
Clearly $u_i\to u$ in $L^1(\Om)$ as $i\to\infty$, and so $w_i\to u$ in $L^1(\Om)$, as desired.
We have by the Leibniz rule of Lemma \ref{lem:Leibniz},
using also the fact that $\Vert Du_i\Vert\le \Vert Du\Vert$ as measures for each $i\in\N$,
\begin{align*}
&\int_{\Om}g_{w_{i}}\,d\mu\\
&  \le \sum_{j=1}^i\int_{\Om}\eta_{ji} g_{\rho_{ji}}\,d\mu
+\int_{\Om}(1-\eta_{i})g_{v_{i}}\,d\mu
+\sum_{j=1}^i\int_{\Om}g_{\eta_{ji}}|\rho_{ji}-v_{i}|\,d\mu\\
&  \le \sum_{j=1}^i\int_{U_{ji}}g_{\rho_{ji}}\,d\mu
+\int_{\Om\setminus (S_1\cup \ldots \cup S_i)}g_{v_{i}}\,d\mu
+\sum_{j=1}^i\int_{U_{ji}}g_{\eta_{ji}}|\rho_{ji}-v_{i}|\,d\mu\\
&  \le \sum_{j=1}^i\left((C_d\alpha^{-1}+1)\Vert Du\Vert(U_{ji})
+\frac{2}{i^2}\right)\\
&\qquad\qquad+\Vert Du\Vert(\Om\setminus (S_1\cup \ldots \cup S_i))+\frac{1}{i}+\frac{3}{i}
\quad\textrm{by }\eqref{eq:estimate of g rho ji},
\eqref{eq:gvi in Om setminus S1 Sn},\eqref{eq:rho ji minus v i condition}\\
&  \le (C_d\alpha^{-1}+1)\sum_{j=1}^i\left(\Vert Du\Vert(S_j)
+\frac{3}{i^2}\right)
+\Vert Du\Vert(\Om\setminus (S_1\cup \ldots \cup S_i))
+\frac{4}{i}\quad\textrm{by }\eqref{eq:Uji and Sj}\\
&  \le \Vert Du\Vert(\Om)+C_d\alpha^{-1}\Vert Du\Vert(S_u)
+\frac{7(C_d\alpha^{-1}+1)}{i}.
\end{align*}
Thus recalling the decomposition of the variation measure
\eqref{eq:variation measure decomposition},
\[
\limsup_{i\to\infty}\int_{\Om}g_{w_{i}}\,d\mu\le \Vert Du\Vert(\Om)
+C_d\alpha^{-1}\Vert Du\Vert(S_u)
=\Vert Du\Vert(\Om)+C_d\alpha^{-1}\Vert Du\Vert^j(\Om),
\]
as desired.

Since we had $|v_{i}-u_i^{\vee}|\le 9/i$ in $\Om\setminus (S_1\cup \ldots \cup S_i)$
(recall \eqref{eq:vi minus u L1})
and $\rho_{ji}\ge u_i^{\vee}-\alpha_i\ge u_i^{\vee}-1/i$ in $U_{ji}$
by \eqref{eq:rho ji pointwise bounds}, it follows that $w_{i}\ge u_i^{\vee}\ge u^{\vee}$
in $\Om$, as desired.
Moreover, since by \eqref{eq:rho ji widetilde estimate in Sj} we have
\[
\rho_{ji}\le\widetilde{\rho}_{ji}\le \max\{u^{\vee},-i\}+\alpha_i\quad\textrm{in }S_j,
\]
then also
\[
w_{i}=\rho_{ji}+\frac{9}{i}\le \max\{u^{\vee},-i\}+\alpha_i+\frac{9}{i}
\le \max\{u^{\vee},-i\}+\frac{10}{i}
\]
in $S_j$, for $j=1,\ldots,i$.
Since $S_u=\bigcup_{j=1}^{\infty}S_j$,
we get $w_i\to u^{\vee}$ in $S_u$.

Finally consider $x\in \Om\setminus S_u$. Fix $\eps>0$.
Recall from \eqref{eq:Wji and Sj} that we had
$U_{ji}\subset W_{ji}$ with
$W_{j1}\supset W_{j2}\supset \ldots\supset S_j$ such that
\[
\capa_1\left(\bigcap_{i=1}^{\infty}W_{ji}\setminus S_j\right)=0.
\]
Denote these $\capa_1$-negligible sets by $H_j$, and $H:=\bigcup_{j=1}^{\infty}H_j$.
Then assume that $x\in \Om\setminus (S_u\cup H)$.

Take $M\in\N$ such that $18/M<\eps$.
Since $x\notin H$, for some $N\in\N$
we have $x\notin U_{1i}\cup \ldots \cup U_{Mi}$ for all 
$i\ge N$.
Now if for a given $i\ge N$ we have
$x\in \Om\setminus \bigcup_{j=1}^{i}U_{ji}$,
then $w_i(x)=v_i(x)+9/i\le u_i^{\vee}(x)+18/i$.
If $x\in \bigcup_{j=1}^{i}U_{ji}$, then $x\in U_{ji}$ for some
$j> M$ and so by \eqref{eq:rho ji pointwise bounds}, $\rho_{ji}(x)\le u_i^{\vee}(x)+18/(j-1)<u_i^{\vee}(x)+\eps$.
Hence using \eqref{eq:vi minus u L1} once more,
\[
w_i(x)=\eta_{ji}(x)\rho_{ji}(x)+(1-\eta_{ji}(x))v_{i}(x)+9/i
\le u_i^{\vee}(x)+\eps+18/i
\]
for all $i\ge N$. Since $\eps>0$ was arbitrary,
we get $w_i(x)\to u^{\vee}(x)$. Thus we have the desired pointwise convergence
$1$-q.e.,
and then in fact we obtain it at every point by redefining the functions $w_i$; recall
\eqref{eq:quasieverywhere equivalence class}.
\end{proof}

In the next example we show that the term $C_a\Vert Du\Vert^j(\Om)$
in Theorem \ref{thm:approximation from above intro} is necessary.

\begin{example}\label{ex:constant necessary}
Let $X=\R$ equipped with the Euclidean metric and the weighted Lebesgue measure
$d\mu:=w\,d\mathcal L^1$, with $w=1$ in $[-1,1]$ and $w=2$ in $\R\setminus [-1,1]$.
Clearly this measure is doubling and the space supports a $(1,1)$-Poincar\'e inequality.
Let $u:=\ch_{[-1,1]}\in\BV(X)$. Let $w\in N^{1,1}(X)$ with $w\ge u^{\vee}=\ch_{[-1,1]}$
everywhere, and let $g$ be an upper gradient of $w$. Let $\eps>0$. For some $x<-1$ we have
\[
1-\eps<|w(x)-w(-1)|\le \int_{x}^{-1} g\,ds.
\]
Similarly for some $y>1$ we have
\[
1-\eps<|w(1)-w(y)|\le \int_{1}^y g\,ds.
\]
Thus
\[
\int_X g\,d\mu\ge \int_{x}^{-1} g\,d\mu+\int_{1}^y g\,d\mu
=2\int_{x}^{-1} g\,ds+2\int_{1}^y g\,ds> 4-4\eps.
\]
Since $\eps>0$ was arbitrary, we get $\Vert g\Vert_{L^1(X)}\ge 4$.
Then the minimal $1$-weak upper gradient of $u$ also satisfies
$\Vert g_u\Vert_{L^1(X)}\ge 4$ (see e.g. \cite[Lemma 1.46]{BB}).
However, defining the Lipschitz functions $u_i(x):=\min\{1,\max\{0,i-i|x|\}\}$, we have 
$u_i\to u$ in $L^1(X)$ and then
\[
\Vert Du\Vert(X)\le \liminf_{i\to\infty}\int_X g_{u_i}\,d\mu
=\liminf_{i\to\infty}\int_{\R} i\ch_{[-1,1]\setminus [-1+1/i,1-1/i]}\,d\mathcal L^1
=2.
\]
This shows that the term $C_a\Vert Du\Vert^j(\Om)$
in Theorem \ref{thm:approximation from above intro} is necessary.
Letting $E:=[-1,1]$, this reasoning also shows that the constant
$C_a$ in Theorem \ref{thm:strict subset theorem intro} is necessary.
\end{example}

\begin{remark}\label{rmk:approximation thm}
	Recall that in the Euclidean setting, the term
	$C_a\Vert Du\Vert^j(\Om)$ is not needed (see
	\cite[Proposition 7.3]{CCM}). Having $\lim_{i\to\infty}\int_{\Om}g_{w_i}\,d\mu= \Vert Du\Vert(\Om)$
	is in fact used in \cite{CCM} to prove the pointwise convergence,
	whereas in our setting it seems necessary to construct
	the approximations ``by hand'', which makes the proof of
	Theorem \ref{thm:approximation from above intro}
	rather technically involved.
	
The assumption $\mu(\Om)<\infty$ in Theorem \ref{thm:approximation from above intro}
is not necessary; it could be removed by using cutoff functions in a very similar way to
\cite[Lemma 3.2]{L-Appr}. We refrain from repeating this rather technical argument here.
\end{remark}

\begin{lemma}\label{lem:norm to uniform convergence}
Let $\Om\subset X$ be open and let $u_i\to u$ in $N^{1,1}(\Om)$.
Then for every $\eps>0$ and every open set $\Om'\Subset\Om$
there exists an open set $G\subset \Om$
such that $\capa_1(G)<\eps$ and $u_i\to u$ uniformly in $\Om'\setminus G$.
\end{lemma}
\begin{proof}
Let $\eps>0$ and let $\Om'\Subset\Om$ be open.
Take $\eta\in \Lip_c(\Om)$ such that $0\le \eta\le 1$ and $\eta=1$ in $\Om'$.
It is easy to check that $\eta u_i\to \eta u$ in $N^{1,1}(X)$, and then
according to \cite[Corollary 1.72]{BB}, there exists an open set $G\subset X$
such that $\capa_1(G)<\eps$ and $\eta u_i\to \eta u$ uniformly in $X\setminus G$.
Since $\eta u_i=u_i$ and $\eta u=u$ in $\Om'$, we have the result.
\end{proof}

In \cite[Proposition 4.1]{KKST3} it is shown that for any $u\in\BV(X)$, there exists a sequence
$(v_i)\subset \liploc(X)$ such that $v_i\to u$ in $L^1(X)$,
\[
\limsup_{i\to\infty}\int_{\Om}g_{v_i}\,d\mu\le C\Vert Du\Vert(\Om),
\]
and
\[
(1-\widetilde{\gamma})u^{\wedge}(x)+\widetilde{\gamma}u^{\vee}(x)\le 
\liminf_{i\to\infty}v_i(x)\le \limsup_{i\to\infty}v_i(x)\le 
\widetilde{\gamma}u^{\wedge}(x)+(1-\widetilde{\gamma})u^{\vee}(x)
\]
for $\mathcal H$-a.e. $x\in X$.
Here $C$ and $0<\widetilde{\gamma}\le 1/2$ are constants that depend only on the
doubling constant of the measure and the constants in the Poincar\'e inequality.
The functions $v_i$ can be taken to be discrete convolution approximations of $u$;
this is a natural approximation method but a drawback is that
in the jump set one does not obtain pointwise convergence, but
rather just the lower and and upper bounds given above.

We can now give a similar result where we do obtain pointwise convergence
$\mathcal H$-almost everywhere also in the jump set.

\begin{corollary}
Let $\Om\subset X$ be open with $\mu(\Om)<\infty$
and let $u\in\BV(\Om)$. Then there exists a sequence
$(v_i)\subset \liploc(\Om)$ such that
$v_i\to u$ in $L^1(\Om)$,
\[
\limsup_{i\to\infty}\int_{\Om}g_{v_i}\,d\mu\le \Vert Du\Vert(\Om)+C_a\Vert Du\Vert^j(\Om),
\]
and $v_i(x)\to u^{\vee}(x)$ for $\mathcal H$-a.e. $x\in\Om$.
\end{corollary}

Here the constant $C_a$ is the same as in Theorem \ref{thm:approximation from above intro}.
The analogous fact naturally holds for $u^{\vee}$ replaced by $u^{\wedge}$.

\begin{proof}
By Theorem \ref{thm:approximation from above intro} we find a sequence
$(w_i)\subset N^{1,1}(\Om)$ such that
$w_i\to u$ in $L^1(\Om)$,
\[
\limsup_{i\to\infty}\int_{\Om}g_{w_i}\,d\mu\le \Vert Du\Vert(\Om)+C_a\Vert Du\Vert^j(\Om),
\]
and $w_i(x)\to u^{\vee}(x)$ for every $x\in\Om$.
By \cite[Theorem 5.47]{BB},
for each $i\in\N$ we find $v_i\in\liploc(\Om)$
such that $\Vert v_i-w_i\Vert_{N^{1,1}(\Om)}<1/i$.
Thus $v_i\to u$ in $L^1(\Om)$ and
\[
\limsup_{i\to\infty}\int_{\Om}g_{v_i}\,d\mu
=\limsup_{i\to\infty}\int_{\Om}g_{w_i}\,d\mu
\le \Vert Du\Vert(\Om)+C_a\Vert Du\Vert^j(\Om).
\]
Take open sets $\Om_1\Subset \Om_2\Subset \ldots \Subset \Om$ with
$\bigcup_{i=1}^{\infty}\Om_i=\Om$.
By
Lemma \ref{lem:norm to uniform convergence} we can assume that
for each $i\in\N$ there is a set $G_i\subset \Om$ such that
$\capa_1(G_i)<2^{-i}$ and $|v_i-w_i|< 1/i$ in $\Om_i\setminus G_i$.
Let $\eps>0$. For sufficiently large $N\in\N$ we have
$\capa_1\left(\bigcup_{j=N}^{\infty}G_j\right)<\eps$, and for every
$x\in \Om\setminus \bigcup_{j=N}^{\infty}G_j$ we have for all $i\ge N$
large enough that $x\in \Om_i$,
\[
|v_i(x)-u^{\vee}(x)|\le |v_i(x)-w_i(x)|+|w_i(x)-u^{\vee}(x)|< 1/i+|w_i(x)-u^{\vee}(x)|\to 0
\]
as $i\to\infty$. Since $\eps>0$ was arbitrary, we have
$v_i(x)\to u^{\vee}(x)$ for $1$-q.e. $x\in\Om$ and then by
\eqref{eq:null sets of Hausdorff measure and capacity}, $\mathcal H$-a.e. $x\in\Om$.
\end{proof}

\noindent Address:\\

\noindent Institut f\"ur Mathematik\\
Universit\"at Augsburg\\
Universit\"atsstr. 14\\
86159 Augsburg, Germany\\
E-mail: {\tt panu.lahti@math.uni-augsburg.de}


\begin{thebibliography}{ACMM}

\bibitem{AH}D. Adams and L. I. Hedberg,
\textit{Function spaces and potential theory},
Grundlehren der Mathematischen Wissenschaften, 314. Springer-Verlag, Berlin, 1996. xii+366 pp.

\bibitem{A1}L. Ambrosio,
\textit{Fine properties of sets of finite perimeter in doubling metric measure spaces},
Calculus of variations, nonsmooth analysis and related topics.
Set-Valued Anal. 10 (2002), no. 2-3, 111--128.

\bibitem{AFP}L. Ambrosio, N. Fusco, and D. Pallara,
\textit{Functions of bounded variation and free discontinuity problems.}
Oxford Mathematical Monographs. The Clarendon Press, Oxford University Press, New York, 2000.

\bibitem{AMP}L. Ambrosio, M. Miranda, Jr., and D. Pallara,
\textit{Special functions of bounded variation in doubling metric measure spaces},
Calculus of variations: topics from the mathematical heritage of E. De Giorgi, 1--45,
Quad. Mat., 14, Dept. Math., Seconda Univ. Napoli, Caserta, 2004.

\bibitem{BB}A. Bj\"orn and J. Bj\"orn,
\textit{Nonlinear potential theory on metric spaces},
EMS Tracts in Mathematics, 17. European Mathematical Society (EMS), Z\"urich, 2011. xii+403 pp.

\bibitem{BB-OD}A. Bj\"orn and J. Bj\"orn,
\textit{Obstacle and Dirichlet problems on arbitrary nonopen sets in metric spaces, and fine topology},
Rev. Mat. Iberoam. 31 (2015), no. 1, 161--214.

\bibitem{BBL-SS}A. Bj\"orn, J. Bj\"orn, and V. Latvala,
\textit{Sobolev spaces, fine gradients and quasicontinuity on quasiopen sets},
Ann. Acad. Sci. Fenn. Math. 41 (2016), no. 2, 551--560.

\bibitem{BBL-CCK}A. Bj\"orn, J. Bj\"orn, and V. Latvala,
\textit{The Cartan, Choquet and Kellogg properties	for the fine topology on metric spaces},
J. Anal. Math. 135 (2018), no. 1, 59--83.

\bibitem{BBL-WC}A. Bj\"orn, J. Bj\"orn, and V. Latvala,
\textit{The weak Cartan property for the p-fine topology on metric spaces},
Indiana Univ. Math. J. 64 (2015), no. 3, 915--941.

\bibitem{BBM}A. Bj\"orn, J. Bj\"orn, and Mal\'y,
\textit{Quasiopen and p-path open sets, and characterizations of quasicontinuity}, 
Potential Anal. 46 (2017), no. 1, 181--199. 

\bibitem{BBS}A. Bj\"orn, J. Bj\"orn, and N. Shanmugalingam,
\textit{Quasicontinuity of Newton-Sobolev functions and density of
Lipschitz functions on metric spaces}, 
Houston J. Math. 34 (2008), no. 4, 1197--1211.

\bibitem{CDLP}M. Carriero, G. Dal Maso, A. Leaci, and E. Pascali,
\textit{Relaxation of the nonparametric plateau problem with an obstacle},
J. Math. Pures Appl. (9) 67 (1988), no. 4, 359--396.

\bibitem{CCM}G. Crasta, V. De Cicco, and A. Malusa,
\textit{Pairings  between  bounded  divergence-measure
vector  fields  and  BV  functions},
preprint 2019.

\bibitem{EvGa}L. C. Evans and R. F. Gariepy,
\textit{Measure theory and fine properties of functions},
Studies in Advanced Mathematics series, CRC Press, Boca Raton, 1992.

\bibitem{Fed}H. Federer,
\textit{Geometric measure theory},
Die Grundlehren der mathematischen Wissenschaften, Band 153 Springer-Verlag New York Inc., New York 1969 xiv+676 pp. 

\bibitem{Fug}B. Fuglede,
\textit{The quasi topology associated with a countably subadditive set function},
Ann. Inst. Fourier 21 (1971), no. 1, 123--169.

\bibitem{Giu84}E. Giusti,
\textit{Minimal surfaces and functions of bounded variation},
Monographs in Mathematics, 80. Birkh\"auser Verlag, Basel, 1984. xii+240 pp.

\bibitem{Hj}P. Haj\l{}asz,
\textit{Sobolev spaces on metric-measure spaces},
Heat kernels and analysis on manifolds, graphs, and metric spaces (Paris, 2002), 173--218,
Contemp. Math., 338, Amer. Math. Soc., Providence, RI, 2003.

\bibitem{HaKi}H. Hakkarainen and J. Kinnunen,
\textit{The BV-capacity in metric spaces},
Manuscripta Math. 132 (2010), no. 1-2, 51--73.

\bibitem{HaSh}H. Hakkarainen and N. Shanmugalingam,
\textit{Comparisons of relative BV-capacities and Sobolev capacity in metric spaces},
Nonlinear Anal. 74 (2011), no. 16, 5525--5543.

\bibitem{Hei}J. Heinonen,
\textit{Lectures on analysis on metric spaces},
Universitext. Springer-Verlag, New York, 2001. x+140 pp.

\bibitem{HKM}J. Heinonen, T. Kilpel\"ainen, and O. Martio,
\textit{Nonlinear potential theory of degenerate elliptic equations},
Unabridged republication of the 1993 original. Dover Publications, Inc., Mineola, NY, 2006. xii+404 pp.

\bibitem{HK}J. Heinonen and P. Koskela,
\textit{Quasiconformal maps in metric spaces with controlled geometry},
Acta Math. 181 (1998), no. 1, 1--61.

\bibitem{HKST15}J. Heinonen, P. Koskela, N. Shanmugalingam, and J. Tyson,
\textit{Sobolev spaces on metric measure spaces.
An approach based on upper gradients},
New Mathematical Monographs, 27. Cambridge University Press, Cambridge, 2015. xii+434 pp.

\bibitem{KiMa}T. Kilpel\"ainen and J. Mal\'y,
\textit{Supersolutions to degenerate elliptic equation on quasi open sets},
Comm. Partial Differential Equations 17 (1992), no. 3-4, 371--405. 

\bibitem{KK}J. Kinnunen and R. Korte,
\textit{Characterizations of Sobolev inequalities on metric spaces},
J. Math. Anal. Appl. 344 (2008), no. 2, 1093--1104.

\bibitem{KKST}J. Kinnunen, R. Korte, N. Shanmugalingam, and H. Tuominen,
\textit{A characterization of Newtonian functions with zero boundary values},
Calc. Var. Partial Differential Equations 43 (2012), no. 3-4, 507--528. 

\bibitem{KKST3}J. Kinnunen, R. Korte, N. Shanmugalingam, and H. Tuominen,
\textit{Pointwise properties of functions of bounded variation in metric spaces},
Rev. Mat. Complut. 27 (2014), no. 1, 41--67.

\bibitem{L-Fed}P. Lahti,
\textit{A Federer-style characterization of sets of finite perimeter on metric spaces},
Calc. Var. Partial Differential Equations, October 2017, 56:150.

\bibitem{L-NC}P. Lahti,
\textit{A new Cartan-type property and strict quasicoverings when $p=1$ in metric spaces},
Ann. Acad. Sci. Fenn. Math. 43 (2018), pp. 1027--1043.

\bibitem{L-FC}P. Lahti,
\textit{A notion of fine continuity for BV functions on metric spaces},
Potential Anal. 46 (2017), no. 2, 279--294.

\bibitem{L-Appr}P. Lahti,
\textit{Approximation of BV by SBV functions in metric spaces},
preprint 2018.
https://arxiv.org/abs/1806.04647

\bibitem{L-DC}P. Lahti,
\textit{Discrete convolutions of BV functions in quasiopen sets
in metric spaces},
preprint 2018.
https://arxiv.org/abs/1812.11087

\bibitem{L-Fedchar}P. Lahti,
\textit{Federer's characterization of sets of finite perimeter in metric spaces},
preprint 2018.
https://arxiv.org/abs/1804.11216

\bibitem{L-LSC}P. Lahti,
\textit{Quasiopen sets, bounded variation and lower semicontinuity in metric spaces},
to appear in Potential Analysis.

\bibitem{L-SA}P. Lahti,
\textit{Strong approximation of sets of finite perimeter in metric spaces},
Manuscripta Math. 155 (2018), no. 3-4, 503--522.

\bibitem{L-CK}P. Lahti,
\textit{The Choquet and Kellogg properties for the fine topology when $p=1$ in metric spaces},
to appear in Journal de Math\'ematiques Pures et Appliqu\'ees.

\bibitem{LaSh}P. Lahti and N. Shanmugalingam,
\textit{Fine properties and a notion of quasicontinuity for $\BV$ functions on metric spaces},
 J. Math. Pures Appl. (9) 107 (2017), no. 2, 150--182.

\bibitem{MZ}J. Mal\'{y} and W. Ziemer,
\textit{Fine regularity of solutions of elliptic partial differential equations},
Mathematical Surveys and Monographs, 51. American Mathematical Society, Providence, RI, 1997. xiv+291 pp.

\bibitem{Maz}V. Maz'ya,
\textit{Sobolev spaces},
Translated from the Russian by T. O. Shaposhnikova. Springer Series in Soviet Mathematics. Springer-Verlag, Berlin, 1985. xix+486 pp.

\bibitem{M}M.~Miranda, Jr.,
\textit{Functions of bounded variation on ``good'' metric spaces},
J. Math. Pures Appl. (9) 82  (2003),  no. 8, 975--1004.

\bibitem{S2}N. Shanmugalingam,
\textit{Harmonic functions on metric spaces},
Illinois J. Math. 45 (2001), no. 3, 1021--1050.

\bibitem{S}N. Shanmugalingam,
\textit{Newtonian spaces: An extension of Sobolev spaces to metric measure spaces},
Rev. Mat. Iberoamericana 16(2) (2000), 243--279.

\bibitem{Zie89}W. P. Ziemer,
\textit{Weakly differentiable functions. Sobolev spaces and functions of bounded variation},
Graduate Texts in Mathematics, 120. Springer-Verlag, New York, 1989. 

\end{thebibliography}
\end{document}